\documentclass{article}
\pdfoutput=1
\usepackage{amsfonts, amsmath, amssymb, amsthm}
\usepackage[toc,page]{appendix}
\usepackage{times}
\usepackage[english]{babel}
\usepackage[shortlabels]{enumitem}
\usepackage{comment}
\usepackage{amsthm}
\usepackage{afterpage}
\usepackage{array}
\usepackage{booktabs}
\usepackage{bm}
\usepackage{bbm}
\usepackage{cases}
\usepackage{soul}
\usepackage{dsfont}
\usepackage{enumitem}
\usepackage{epsfig}
\usepackage{eucal}
\usepackage{float}
\usepackage{fontenc}
\usepackage{graphicx}
\usepackage{mathbbol}
\usepackage{rotating}
\usepackage[hmargin=2cm,vmargin=3cm]{geometry}
\usepackage{textcomp}
\usepackage{amstext}
\usepackage{setspace}
\usepackage{multicol}
\usepackage{multirow,bigdelim}
\usepackage{subfigure}
\usepackage{wasysym}
\usepackage{babel}
\usepackage{natbib}
\usepackage{lscape}
\usepackage{mathrsfs}

\usepackage{float}
\restylefloat{table}
\usepackage{mdframed}

\newtheorem{theorem}{Theorem}

\newtheorem{example}{Example}

\newtheorem{corollary}[theorem]{Corollary} 
\newtheorem{lemma}{Lemma}
\newtheorem{remark}{Remark}

\def\Er{\text{Er}}
\def\RV{\text{RV}}

\def\NGG{\text{NGG}}

\def\red#1{{#1}}%

\newcommand{\dif}{\ensuremath{\mathrm{d}}}

\newcommand{\T}{\ensuremath{\mathrm{\scriptscriptstyle T}}}
\usepackage{bm}

\newcommand{\alphab}{\ensuremath{\bm\alpha}}

\newcommand{\etab}{\ensuremath{\bm\eta}}
\newcommand{\thetab}{\ensuremath{\bm\theta}}

\newcommand{\sigmab}{\ensuremath{\bm\sigma}}

\DeclareMathOperator{\E}{E}

\DeclareMathOperator{\iid}{\overset{iid}{\sim}}

\DeclareMathOperator{\Beta}{Beta}

\DeclareMathOperator{\Ga}{Gamma}

\begin{document}

\title{\bf Heavy-Tailed NGG-Mixture Models}
  \author{Vianey Palacios Ram\'irez,  
    Miguel de Carvalho, Luis Guti\'errez Inostroza\footnote{
     V.~Palacios Ram\'irez is Lecturer in Statistics, Newcastle University (\textit{Vianey.Palacios-Ramirez@newcastle.ac.uk}). M.~de Carvalho is Professor in Statistics, University of Edinburgh (\textit{Miguel.deCarvalho@ed.ac.uk}). L.~Guti\'errez Inostroza is Assistant Professor in Statistics, Pontificia Universidad Cat\'olica de Chile. ANID--Millennium Science Initiative Program--Millennium Nucleus Center for the Discovery of Structures in Complex Data. (\textit{llgutier@mat.uc.cl}).}}
     \date{}
  \maketitle

\begin{abstract}
Heavy tails are often found in practice, and yet they are an Achilles 
heel of a variety of mainstream random probability measures such as
the Dirichlet process (DP). The first contribution of this paper focuses on
characterizing the tails of the so-called normalized generalized gamma (NGG) process.
We show that the right tail of an  NGG process is heavy-tailed provided that the centering
distribution is itself heavy-tailed; the DP is the only member of the NGG class 
that fails to obey this convenient property. A second contribution of the
paper rests on the development of two classes of heavy-tailed mixture
models and the assessment of their relative merits.  Multivariate
extensions of the proposed heavy-tailed mixtures are devised here, 
along with a predictor-dependent version,  to learn about the effect
of covariates on a multivariate heavy-tailed response.  The simulation
study suggests that the proposed method performs well in various 
scenarios, and we showcase the application of the proposed methods in
a neuroscience dataset.
\end{abstract}


\section{Introduction}\label{introduction} 
Thousands of heavy-tailed signals are produced on a day-to-day basis
across the globe in fields as diverse as engineering, finance, and
medicine. And yet, despite the widespread need for modeling these,
heavy tails remain a weak spot of several established random
probability measures such as the Dirichlet process (DP)
(e.g. Section~4.3.\cite{ghosal2017}).

Prior to introducing the main problems to
be addressed,  and the main contributions of this paper, we first lay the groundwork. The normalized generalized gamma (NGG) process is a random probability measure that was introduced and investigated by \cite{lijoi2007} and that has received considerable attention in recent years ( e.g., \cite{james2009,lijoi2010,barrios2013,favaro2016}). The NGG class includes the Dirichlet process, the stable process, and the normalized inverse Gaussian process as particular cases. NGGs are built by normalizing the generalized gamma process, and can be understood as a particular case of a normalized random measure with independent increments (NRMI) \cite{regazzini2003, lijoi2007, james2009, lijoi2010, barrios2013}. Some Bayesian nonparametric approaches, such as the NGG process, can be understood as an extension of standard parametric methods in the sense that they are centered a priori around a parametric model, $\{G_0 \equiv G_{0, \thetab}: \thetab \in \Theta \subseteq \mathbb{R}^q\}$, but assign positive mass to a variety of alternatives. Thus, a recurring theme in much of the Bayesian nonparametric literature is to regard a parametric approach---known as the baseline or centering distribution---as a reference, while allowing for deviations from it. See the monographs of \cite{muller2015} and \cite{ghosal2017} or the review paper of \cite{muller2013} for an introduction to Bayesian nonparametric inference.  

It is well-known that the tails of the Dirichlet process are exponentially much thinner than those of the baseline \cite[Section~4.2.3]{ghosal2017}. Motivated by this, this paper opens with the question of whether this is simply a property of the Dirichlet process or is more generally an attribute of the NGG process. 
Hence, the first contribution of this paper will focus on the characterization of the tails of the NGG process, and we will derive envelopes for the trajectories of the tail of the process. As will be discussed below (Section~\ref{pitman}), the latter envelopes combined with those of \cite{doss1982} offer a complete portrait of the tails of the NGG process. In addition, we then show that the tail of the NGG process is only  moderately thinner than that of the baseline, except for the Dirichlet process. In particular, our results imply that, with the exception of the Dirichlet process, the tail of the NGG process is heavy-tailed, provided that the baseline is itself heavy-tailed. This result sharply contrasts with the Dirichlet process, given that, as mentioned earlier, its tail is exponentially much lighter than that of the centering. Such property of the DP may seem unexpected, given that the process is centered around $G_0$. Yet, a notable example illustrates this: the tails of the DP centered around a Cauchy distribution are almost exponential, but $G_0$ has no mean \citep[][
Example~4.24]{ghosal2017}.

A second contribution of the paper rests on the study of two
classes of heavy-tailed mixture models and the assessment of their
relative merits. The heavy-tailed mixture models devised here have
links with the phase-type scale mixtures of \cite{bladt2018} and the
infinite mixtures of Pareto distributions of \cite{tressou2008}.  Our
focus differs, however,  from these papers in several important
ways.  Some key differences are that we take a general view of 
heavy-tailed NGG-mixtures and take advantage of our novel 
characterization of its tail. In addition, by keeping a
general focus in mind, our theoretical and numerical analyses will reveal that there are some good reasons for preferring NGG scale mixtures over NGG-mixtures built from heavy-tailed kernels. Finally, motivated by the fact that heavy-tailed data are frequently multivariate---and since covariates 
are often available---we further extend the proposed heavy-tailed scale mixture models to model these as well. 
In other words, multivariate extensions of the proposed heavy-tailed mixtures are also devised below, along with a predictor-dependent version to learn about the effect of covariates on a multivariate heavy-tailed response. Our theoretical and numerical analyses pinpoint a clear preference for NGG scale mixtures over Dirichlet process mixtures of heavy-tailed kernels; overall, NGG scale mixtures tend to have superior numerical performance in the bulk and tails.

A final comment on the jargon of heavy tails is in order. Following the standard convention in the literature on heavy tails \citep[e.g.,][]{resnick2007}, here we will characterize these via regular variation \citep{bingham1989}. A distribution
function $F(y) = P(Y \leq y)$, or its density $f = \dif F / \dif y$ in
case it exists, is said to have a regularly varying tail,
with tail index $\alpha \equiv \alpha(F) > 0$, if
\begin{equation}\label{rvdef} \lim_{y \to \infty} \frac{P(Y > yt)}{P(Y
> y)} = t^{-\alpha}.
\end{equation} The smaller the tail index, the slower the decay of the tail, $1
- F(y)$, to 0 as $y \to \infty$, and thus the more heavy-tailed is the
distribution. Throughout, the notation $1 - F \in \RV_{-\alpha}$ is
used to denote that $F$ verifies \eqref{rvdef}.

The remainder of this paper unfolds as follows. In Section~\ref{sect:the_model}, we study the tails of the NGG process
and construct two classes of heavy-tailed NGG mixture models. In Section~\ref{sect:extensions},  we expand the proposed toolbox to the multivariate setting as well as to a conditional framework. Section~\ref{simulation} illustrates the performance of the proposed methods and reports the main findings of our numerical studies. An application of the proposed methods to a neuroscience case study is given in Section~\ref{application}. Finally, in Section~\ref{discussion}, we present closing remarks. Proofs are available in the Appendix, and further technical details and supporting numerical evidence can be found in the online supplementary material; the R~package~\texttt{NGGR}, available from the supplementary material, implements instances of the methods proposed herein. \vspace{-0.4cm}

\section{Heavy-tailed NGG-mixture models}\label{sect:the_model}
\subsection{On the tails of the NGG process}\label{pitman}\vspace{+.1cm} 
\noindent \textbf{Subordinator representations}: 
We start by studying the tails of the NGG process since its properties will be vital for constructing our class of mixture models. The main result of this section is Theorem~\ref{tailNGG}, but before we can  discuss its implications, we first lay the groundwork. 
A subordinator, $\{S(t): t \geq 0\}$, is an increasing stochastic process over the positive real line that has independent and homogeneous
increments \cite[e.g.,][Chapter~1]{applebaum2009}. By the so-called L\'evy--Khintchine representation \citep[e.g.,][Section~1.2]{bertoin1999}, a subordinator is fully characterized by its Laplace exponent, 
\begin{equation*}
  \Phi(\lambda) = \texttt{k} + \texttt{d}\lambda + \int_{0}^{\infty} (1 - e^{-\lambda\, u}) \nu(\dif u),
  \quad \lambda \geq 0, 
\end{equation*}
that obeys $\E[\exp\{-\lambda S(t)\}] = \exp\{-t \Phi(\lambda)\}$, for $t \geq 0$; here, $\texttt{k} > 0$ is the killing rate, $\texttt{d} > 0$ is the drift coefficient, and $\nu$ is a measure on $(0, \infty)$, known as L\'evy measure, that governs the law of the increments and which obeys the constraint $\int_{0}^{\infty} \min(1, u) \, \nu(\dif u) < \infty$.

It is well known that the NGG process, introduced in \cite{lijoi2007}, inherently admits a subordinator representation due to its construction through the normalization of the generalized gamma process. A random probability measure $G$ follows an NGG process, here denoted as $G \sim \text{NGG}(M, \tau, D, G_0)$, if 
\begin{equation}\label{nngg}
  G(y) =\frac{S\{MG_0(y)\}}{S(M)}, \quad y \in \mathbb{R}.
\end{equation}
Here, $S$ is a generalized gamma subordinator, that is, $\texttt{k} = \texttt{d} = 0$ and $$\nu(\dif u) = D / \Gamma(1-D) u^{-1-D} \exp(-\tau u) \, \dif u,$$ for $u > 0$, where $\Gamma(z) = \int_0^{\infty} u^{z - 1} \exp(-u)\, \dif u$ is the gamma function. With a slight abuse of notation throughout, we use $G$ to denote both the random measure of an NGG process and its corresponding distribution function. 

The parameter $G_0$ of the NGG is known as the centering distribution function, and the other parameters are subject to the constraints $M > 0$, $\tau \geq 0$, and $0 \leq D < 1$. Particular cases include the Dirichlet process, NGG$(M, 1, 0, G_0)$, the stable process, NGG$(1, 0, D, G_0)$, and the normalized inverse Gaussian NGG$(1, \tau, 1/2, G_0)$ \citep[][]{lijoi2005, lijoi2007}; the Dirichlet process can also be obtained for any $\tau > 0$ and not just for $\tau = 1$, provided $D = 0$. Keeping this in mind, and the fact that in our context it will be important to separate the Dirichlet process from the other members of the NGG class, we introduce the following notation:
$$
\begin{cases}
\mathcal{D} = \{(M, \tau, D): M > 0, \tau > 0, D = 0\}, \\
\mathcal{N} = \{(M, \tau, D): M > 0, \tau \geq 0, D \in (0, 1)\}.  
\end{cases}
$$
The general subordinator representation of the NGG process is given in \eqref{nngg}. In the particular case $G \sim \NGG(M, \tau, D, G_0)$ with $(M, \tau, D) \in \mathcal{D}$, then 
\begin{equation}\label{repdp}
  G(y) = \frac{\gamma\{M G_0(y)\}}{\gamma(M)}, \quad y \in \mathbb{R},
\end{equation}
where $\gamma$ is a gamma process, that is, 
$\texttt{k} = \texttt{d} = 0$ and $\nu(\dif u) = u^{-1} \exp(-u) \, \dif u$, for $u > 0$.
Finally, we recall that if $G$ is a NGG process, then it admits a stick-breaking representation 
\begin{equation}\label{stick}
  G = \sum_{h=1}^{\infty} \pi_h \delta_{Y_h},  \quad Y_h \iid G_0,
\end{equation}
with a stick-breaking sequence $V_h$ as in \citet[Proposition 3]{favaro2016},  
so that  $\pi_1 = V_1$ and $\pi_h = V_h \prod_{k < h}(1 - V_k)$, for $h = 2, 3, \dots,$ and where $\delta_Y$ denotes a point mass at $Y$.

\noindent \textbf{On the tails of the NGG process}:
We now examine the behavior of $1 - G(y)$, as $y$ approaches the right endpoint, ${y^*} = \sup\{y: G(y) < 1\}$. Recall that both $G \sim \text{NGG}(M, \tau, D, G_0)$ and $G_0$ are supported over the same set, and thus the right endpoints of $G$ and $G_0$ coincide. Following the standard convention in extreme value analysis, we focus on the right tail, but all claims below apply to the left tail with minor adjustments. 

The tails of the Dirichlet process are much lighter than those of the centering distribution, with probability one, a fact that can be shown using the subordinator representation in~\ref{repdp}. Theorem~\ref{tailsdp} is well known since \cite{doss1982}, it formalizes the latter claim, and it is only included here for completeness.
\begin{theorem}[Tails of $\NGG$ in $\mathcal{D}$]\label{tailsdp} 
  Let $G(y)$ be the distribution of a $\NGG(M, \tau, D, G_0)$ process with $(M, \tau, D) \in \mathcal{N}$ and {with non-atomic $G_0$}. Then, 
  \begin{equation*}
    \begin{split}
    \lim \inf_{y \to {y^*}} \frac{1 - G(y)}{g_r\{1 - G_0(y)\}} = 
    \begin{cases}
      0, & \text{if }s < 1, \\
      \infty, & \text{if }s > 1, \\ 
    \end{cases} \quad a.s.,  \qquad \\
    \lim \sup_{y \to {y^*}} \frac{1 - G(y)}{h_r\{1 - G_0(y)\}} =
    \begin{cases}
      0, & \text{if }r > 1, \\
      \infty, & \text{if }r \leq 1, \\ 
    \end{cases} \quad a.s., \qquad
    \end{split}
  \end{equation*}
  with $g_r(t) = \exp\{-r\log|\log t|/t\}$ and $h_r(t) = \exp\{- 1 / (t |\log t|^{r})\}$, for $0 < t < 1$.
\end{theorem}
\begin{proof}
See \cite{doss1982} or \citet[][Theorem~4.22]{ghosal2017}.
\end{proof}

\noindent The takeaway from Theorem~\ref{tailsdp} can be loosely summarized in bounds, for $r > 1$, a.s.~eventually for a large $y$, 
\begin{equation}\label{ineqdp}
  \resizebox{0.92\hsize}{!}{%
    $\exp\bigg[-\frac{r \log | \log M \{1 - G_0(y)\}|}{M \{1 - G_0(y)\}} \bigg]\leq 1 - G(y) \leq \exp\bigg[- \frac{1}{M \{1 - G_0(y)\} |\log M \{1 - G_0(y)\}|^r} \bigg],$%
  }
\end{equation}
where $G\sim \NGG(M, \tau, D, G_0)$ for $(M, \tau, D) \in \mathcal{D}$. Hence, the tails of the Dirichlet process are almost exponential, even if $G_0$ is heavy-tailed. Despite being one of the most popular Bayesian nonparametric priors, such deficiency of the DP rules out its use when the goal is to model heavy tails, extreme values, and risk. As shown next, what happens with the NGG process over $\mathcal{N}$ is substantially different as its tails are close to those of the baseline. 

\begin{theorem}[Tails of $\NGG$ in $\mathcal{N}$]\label{tailNGG} 
  Let $G(y)$ be the distribution of an $\NGG(M, \tau, D, G_0)$ process with $(M, \tau, D) \in \mathcal{N}$. Then, 
  \begin{equation*}
    \begin{split}
    &\lim \inf_{y \to y^*} \frac{1 - G(y)}{l\{M(1 - G_0(y))\}} = D(1 - D)^{(1 - D)/D}/S(M) 
      \quad a.s.,
      \\
    &\lim \sup_{y \to y^*} \frac{1 - G(y)}{u_r\{M(1 - G_0(y))\}} =
    \begin{cases}
      0, & r > 1, \\
      \infty, & r \leq 1, \\ 
    \end{cases} \quad a.s., \qquad
    \end{split}
  \end{equation*}
  with $l(t)  = t^{1/D}\log |\log t| /\{(\log |\log t|+\tau^Dt)^{1/D}-\tau t^{1/D}\}$, for $0 < t < e^{-1}$, and $u_{r}(t) = t^{1/D} |\log t|^{r/D}$, for $0 < t < e^{-r}$.
\end{theorem}
\noindent Theorem~\ref{tailNGG} warrants some remarks. A key takeway is that the tails of a random distribution following a $\NGG$ over $\mathcal{N}$ are almost as heavy as those of the centering, $G_0$. Indeed, the rough takeaway of Theorem~\ref{tailNGG} is that a.s.~eventually for a large $y$, 
\begin{equation}\label{ineqNGG}
\begin{split}
\frac{1}{S(M)} D(1 - D)^{(1 - D)/D}l(M\{1 - G_0(y)\}) \leq 1 - G(y) \leq u_r(M\{1 - G_0(y)\}),
\end{split}
\end{equation}
for $r > 1$ with $y > G_0^{-1}(1-e^{-r})$. 
Since $S(M)$ is random, the na\"ive lower bound in \eqref{ineqNGG} would vary with each realization of the subordinator. To mitigate this, we use well-known results on long-run behavior of subordinators  \citep[e.g.,][p.~92]{bertoin1996} which yield a.s.~eventually for a large $y$, 
\begin{equation}\label{ineqNGGnewbound}
\frac{D(1 - D)^{(1 - D)/D}}{(M^*)^{1/D}\log(M^*)^{r/D}}\, l(M\{1 - G_0(y)\}) \leq 1 - G(y) \leq u_r(M\{1 - G_0(y)\}),
\end{equation}
for $r > 1$, $M^*\gg e^{-r}$ and $M^* > M$, with $y > G_0^{-1}(1-e^{-r})$. Technical details on the derivation of \eqref{ineqNGGnewbound} are in the supplementary material (Section 2); this refined lower bound follows by determining an asymptotic upper bound for $S(M)$, for a large $M^* > M$ (e.g., in our experience setting $M^*=10$ already gives a reasonable bound, for any $M < 10$).  Numerical illustrations of the asymptotic envelopes in \eqref{ineqNGGnewbound} are presented in Figure~\ref{comparison} {for a specific instance of the NGG process in $\mathcal{N}$, the stable process}, and further examples are included in the supplementary material. Another implication of Theorem~\ref{tailNGG} is that if the tail of the centering distribution of NGG in $\mathcal{N}$ is heavy-tailed, then so will be that of the corresponding process, though with a lighter tail. 

\begin{corollary}[Stability of the heavy-tail property in $\mathcal{N}$]\label{cor} 
 If $G \sim \NGG(M, \tau, D, G_0)$, with $(M, \tau, D) \in \mathcal{N}$, and $G_0$ has a regularly varying tail, with tail index $\alpha_0 \equiv \alpha(G_0) > 0$, then $G$ has a regularly varying tail, with tail index $\alpha(G) = \alpha(G_0) / D$,  {almost surely}.
 \end{corollary}
 \noindent Corollary~\ref{cor} does not conflict with the full support of NGGs \citep[e.g.,][]{bissiri2014}. Indeed, while the NGG centered on a regularly varying tail is heavy-tailed with probability one, the support of the NGG remains the space of all probability measures---provided the baseline measure satisfies the requirements stated in \cite{bissiri2014}.

 Next, we compare Theorems~\ref{tailsdp} and \ref{tailNGG} in one example. 

\begin{example}[Pareto centering distribution: $\mathcal{D}$ \textit{versus} $\mathcal{N}$]\label{paretoex} \normalfont 
    Suppose first that $G \sim \text{NGG}(M, 1, 0, G_0)$ where $G_0$ is a standard Pareto distribution, that is, $1 - G_0(y) = 1 / y$, for $y > 1$. Then \eqref{ineqdp} yields
  \begin{equation}\label{ineqdpp_pareto}
    \exp[-r y/M \log|\log (M / y)|\,]\leq 1 - G(y) \leq \exp[-y/\{M |\log (M / y)|^r\}],
  \end{equation}
  for $r > 1$, and hence the tails of the Dirichlet process $G \sim \text{NGG}(M, 1, 0, G_0)$ are almost exponential, despite the fact that $G_0$ is heavy-tailed. Indeed, as shown in the supplementary material (Section~2), the two bounds in \eqref{ineqdpp_pareto} are in the Gumbel maximum domain of attraction, although $G_0$ is in the Fr\'echet maximum domain of attraction. Suppose now that $G \sim \text{NGG}(M, \tau, D, G_0)$, with $(M, \tau, D) \in \mathcal{N}$. Then,
  \begin{equation*}\label{ineqsp_pareto}
    \resizebox{0.95\hsize}{!}{%
      $\frac{\red{C}\,y^{-1/D}M^{1/D} \log |\log My^{-1}|}{(\log |\log My^{-1}|+\tau^DMy^{-1})^{1/D}-\tau (My^{-1})^{1/D}} \leq 1 - G(y) \leq y^{-1/D}M^{1/D}|\log My^{-1}|^{r/D},$%
    }
  \end{equation*}
  \red{for $C=1/\{(M^*)^{1/D}\log(M^*)^{r/D}\}$, $r > 1$, $M^*\gg0$ and $M^* > M$.}
  It follows from \red{Corollary~\ref{cor}} that $1 - G(y)$ is regularly varying at infinity with tail index $1/D$. Figure~\ref{comparison} illustrates that, for example, as predicted by Theorem~\ref{tailNGG}, the trajectories of the stable process, $\NGG(1, 0, 0.5, G_0)$, follow the asymptotic envelopes in \eqref{ineqsp_pareto}; in addition, the same figure illustrates that the envelopes of the Dirichlet process in \eqref{ineqdp} fall abruptly in comparison with those of the stable process.
  \end{example}

\begin{figure}
  \centering
  \includegraphics[scale = 0.5]{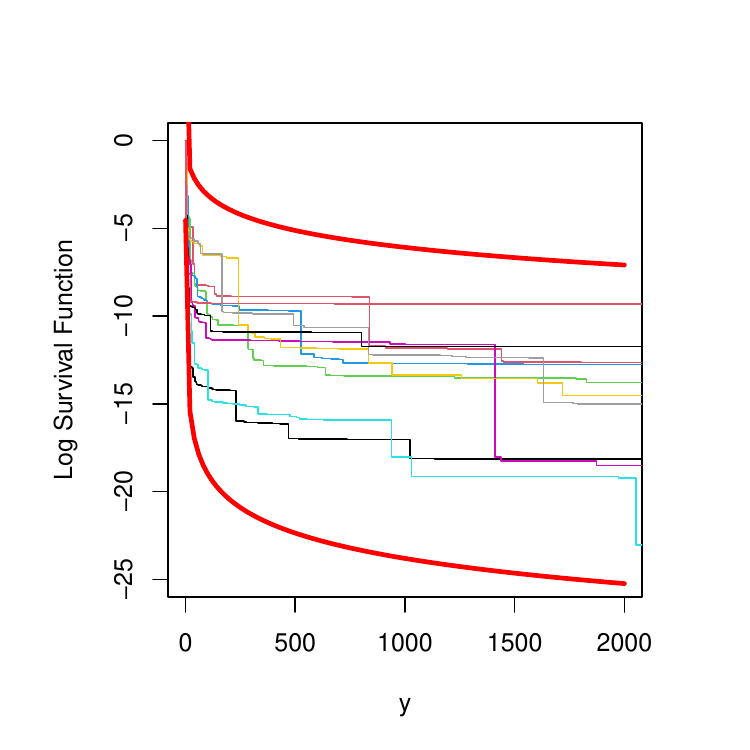} 
  \includegraphics[scale = 0.5]{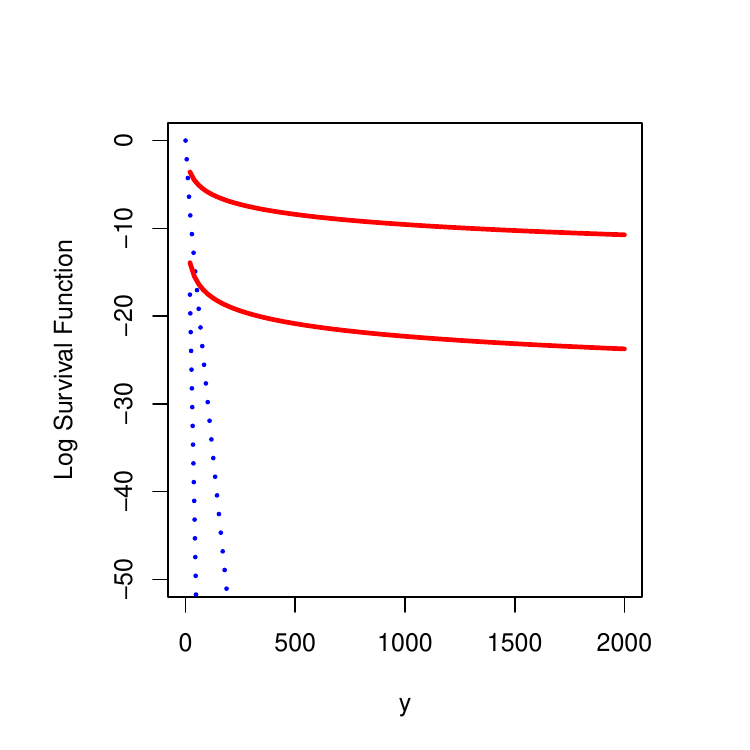}
  \caption{ \label{comparison} Asymptotic envelopes for Example~\ref{paretoex}. Left: Asymptotic envelopes that follow from Theorem~\ref{tailNGG} along with random trajectories of the log survival functions from a stable process $\NGG(1, 0, 0.5, G_0)$. Right: The same envelopes for log survival function of a stable process $\NGG(1, 0, 0.5, G_0)$ (solid) against those of a Dirichlet process $\NGG(1, 1, 0, G_0)$ (dashed); $G_0$ is the standard unit Pareto distribution.}
\end{figure}
\subsection{Heavy-tailed NGG-mixture models}\label{mixture}

Empowered by the main findings from Section~\ref{pitman}, this section shows that two classes of heavy-tailed NGG mixture models can be devised, and it discusses the relative merits of each option.
Below, we focus on univariate mixtures; comments on multivariate and conditional extensions are given in  Section~\ref{sect:extensions}.

\noindent \textbf{Heavy-tailed NGG-$\mathcal{N}$ scale mixtures}:  
The first class of mixture models to be considered is 
\begin{equation}\label{pymixtures}
  \begin{cases}
   f(y) = \int_{0}^{\infty} K_\sigma(y; \eta_\sigma) \, \dif G(\sigma), \\
   G \sim \text{NGG}(M, \tau, D, \red{G_0}), \quad (M, \tau, D) \in \mathcal{N}\\ 
   {1 - G_0(\sigma) = \frac{\mathscr{L}(\sigma)}{\sigma^{\alpha_0}}}, 
  \end{cases}
\end{equation}
with $y \in \mathbb{R}$ and $\alpha_0 \equiv \alpha(G_0) > 0$. Here, $K_\sigma(\cdot) = K(\cdot/\sigma; \eta_\sigma)/\sigma$ with $K$ being a kernel, $\sigma > 0$ is a scale parameter, $\eta_\sigma$ denotes additional parameters (possibly related with $\sigma$), and $\mathscr{L}$ is a slowly varying function, that is, $\mathscr{L}(yt) / \mathscr{L}(y) \to 1$ as $y \to \infty$ for any $t > 0$. \vspace{0.2cm}

\noindent \textbf{Heavy-tailed NGG shape mixtures}: The second class of mixture models to be considered is 
\begin{equation}\label{pymixtures2}
  \begin{cases}
   f(y) = \int_{0}^{\infty} \mathbb{k}(y; \alpha, \eta_\alpha) \, \dif G(\alpha), \\
    G \sim \text{NGG}(M, \tau, D, \red{G_0}),  \quad (M, \tau, D) \in \mathcal{N} \cup \mathcal{D}, \\
    1 - \mathbb{K}(y; \alpha, \eta_\alpha) = \frac{\mathscr{L}(y)}{y^{\alpha}},
  \end{cases}
\end{equation}
with $y \in \mathbb{R}$ and $\alpha \geq 0$.  Here, $\mathbb{k}$ is a Pareto-type kernel with distribution function $\mathbb{K}$, $\eta_{\alpha}$ denotes additional parameters (possibly related with $\alpha$), and $\mathscr{L}$ is a slowly varying function. \vspace{0.2cm}

\noindent The following theorem implies that NGG mixture models in \eqref{pymixtures} and \eqref{pymixtures2} are indeed heavy-tailed. In addition, it shows how their tail indices relate to that of the centering. Below, $(a)_+ = \max(a, 0)$ denotes the positive part function.

\begin{theorem}[Heavy-tailed NGG-mixtures]\label{props} The following results hold { for $F$ the distribution function of $f$ in \eqref{pymixtures} and \eqref{pymixtures2}}:
  \begin{enumerate}[a)]
  \item  {If \eqref{pymixtures} holds, with {$U \sim K_\sigma(\cdot; \eta_\sigma)$}, $\E(U_+^{\alpha_0}) < \infty$, $P(U_+ > \sigma) = o\{1 - G_0(\sigma)\}^{1/D}$ and $\lim \inf_{\sigma \to \infty} \mathscr{L}(\sigma) > 0$, then the tail of $f$ is regularly varying with tail index $\alpha(F) = \alpha_0 / D$, {almost surely}.}
  \item   If \eqref{pymixtures2} holds,   
 then the tail of $f$ is regularly varying with tail index $\alpha(F) = \inf\{\alpha: G_0(\alpha) > 0\}$, for any $G \sim \text{NGG}(D, M, G_0(\alpha))$, {almost surely}.
  \end{enumerate}
\end{theorem}
\noindent Some comments on Theorem~\ref{props} are in order:
\begin{itemize}
\item Theorem~\ref{props} a) shows that if the centering of the NGG process is heavy-tailed, and if the tail of the kernel is not heavier than that of the mixing, then $f$ in \eqref{pymixtures} is heavy-tailed---with the same tail index as that of the mixing; that is, $\alpha(F) =  \alpha(G) = \alpha_0/D$ {with probability one}, where $\alpha_0 = \alpha(G_0)$. Interestingly, this result offers a partial answer to an insightful open problem raised by \citet[][Theorem~3.5]{li2019} on the range of $\alpha(F)$ under a polynomial decay of $G_0$. Indeed, Theorem~\ref{props} a) illustrates that under a polynomial decay of $G_0$, it holds that $\alpha(F) = \alpha_0/D$ differs from $0$ and $\infty$, almost surely, and thus $\alpha(F)$ for \eqref{pymixtures} does not concentrate on a singleton {when assigned hyperpriors on $D$ an $\alpha_0$}. The range of $\alpha(F)$ can be the positive real line, provided that prior densities of $D$ and $\alpha_0$ put positive mass on all $D$ and $\alpha_0$, that is, $\textsf{p}(D) > 0$ or $\textsf{p}(\alpha_0) > 0$.  
\item Theorem~\ref{props} b) is a folklore result that formalizes the idea that infinite shape mixtures of heavy-tailed kernels are themselves heavy-tailed, with a tail index equal to that of the heaviest component (i.e.,~equal to the left endpoint of the centering in the case of \eqref{pymixtures2}). While this would seem like a standard result, we could not find a formal statement of the result in the literature. As evident from the proof, the argument generally holds for any stick-breaking process $G$ and not just for NGG processes. 
\end{itemize}

\noindent For simplification, throughout, we refer to \eqref{pymixtures} as scale mixtures, but as it will be shown below,  \eqref{pymixtures} also includes scale--shape mixtures of light-tailed kernels. Next, we present instances of the specifications in \eqref{pymixtures} and \eqref{pymixtures2} that showcase the generality of the latter and how they relate to some mainstream approaches. 

\begin{example}[NGG scale mixtures with an Erlang kernel] \label{erlangex}\normalfont
  As an example of \eqref{pymixtures}, consider 
  \begin{equation}\label{erlang1}
    f(y) = \sum_{h=1}^{\infty} \pi_h \, \Er(y; a,  \sigma_h), 
  \end{equation}
  where $\Er(y; a, \sigma_h)$ is the density of the Erlang distribution with shape $a \in \mathbb{N}$ and scale $\sigma_h > 0$, and $\pi_h = V_h \prod_{k < h}(1 - V_k)$ with $V_h \overset {\text{ind}}{\sim} \Beta(1 - D, \red{h D})$, for $h \in \mathbb{N}.$ In the notation of  \eqref{pymixtures}, with $K(y ; \eta_\sigma) = \text{Er}(y; a, 1)$ and $\eta_\sigma = a$. For this kernel, the infinite mixture in \eqref{pymixtures} shows a connection with the phase-type scale mixtures of \cite{bladt2018} and \cite{ayala2022}. Two key differences are that: \emph{i}) since the mixing in \eqref{pymixtures} is over an NGG process, inference can be conducted by standard Bayesian nonparametric samplers, whereas fitting phase–type scale mixtures is far from  straightforward; \emph{ii}) Theorem~\ref{props} a) allows for a general kernel with a lighter tail than that of the mixing, whereas \cite{bladt2018} consider phase-type kernels. Another version of \eqref{pymixtures}, along the same lines as \eqref{erlang1}, which we have found to perform well in practice, is the following scale--shape mixture of NGG,
  \begin{equation}\label{erlang2}
    f(y) = \sum_{h=1}^{\infty} \pi_h \, \Er(y; \lceil \sigma_h \rceil, \sigma_h/\lambda),
  \end{equation} 
  where $\lceil \cdot \rceil$ is the ceiling function, and $K(y ; \eta_\sigma) = \text{Er}(y; \lceil \sigma \rceil, 1 / \lambda)$ with $\eta_\sigma = (\lceil \sigma \rceil, 1 / \lambda)^{\T}$ in the notation of \eqref{pymixtures}. We will revisit \eqref{erlang2} later in the paper.
\end{example}

\begin{example}[NGG shape mixtures with a Pareto-type kernel]\normalfont \label{exptk}
Since the Burr, $F$, and generalized Pareto distributions are particular cases of Pareto-type kernels (Table~\ref{exrvs}), the mixture model in \eqref{pymixtures2} includes as particular cases infinite mixtures of such distributions with a NGG mixing. In addition, \eqref{pymixtures2} also includes the Pareto kernel Dirichlet process mixtures of \cite{tressou2008}. \noindent  
\end{example}

\begin{table}
  \caption{\label{exrvs} Instances of Pareto-type kernels $\mathbb{k}$ following \eqref{pymixtures2}.} 
\centering
\footnotesize
\begin{center}
\begin{tabular}{lllll}
  \hline \hline
  \textbf{Distribution} & \textbf{Kernel} ($\mathbb{k}$) & \textbf{Slowly varying function} ($\mathscr{L}$) & \textbf{Tail index} ($\alpha_{\mathbb{K}}$) \\
  \hline
  Burr & $\propto y^{c - 1}/(1 + y^c)^{a + 1}$ & $\propto (y^{-c} + 1)^{-(a + 1)}$ & $ca$ \\
  $F$ & $\propto y^{a/2 - 1}(a + by)^{-(a + b)/2}$ & $\propto (a / y + b)^{-(a + b) / 2}$ & $b/2$ \\
  Generalized Pareto & $\propto (1 - ay/\sigma)^{1/a - 1}$& $\propto (1/y - k / \sigma)^{1/a - 1}$& -$1/a$  \\
  Pareto & $\propto y^{-(a+1)}$ & $\propto 1$ & $a$ \\  
  Student-$t$ & $\propto (1 + y^2/a)^{-(a + 1)/ 2}$ & $\propto (1 / y^2 + 1 / a)^{-(a+1) / 2}$ & $a$ \\
  \hline \hline 
\end{tabular}
\end{center}
\end{table}

There are some reasons for preferring the NGG scale mixtures in \eqref{pymixtures} over \eqref{pymixtures2}. In particular, scale mixtures of stable processes offer a more natural link between the tail of the centering and the tail of $F$ than \eqref{pymixtures2}. For example, if in \eqref{pymixtures} the centering is a Pareto Type II distribution over $(0, \infty)$ (i.e., $1 - G_0(\sigma) = (1 + \sigma)^{-\alpha_0}$) then $\alpha(F) = \alpha_0 / D$, and hence both the centering and $F$ are heavy-tailed. Yet, if the centering in \eqref{pymixtures2} is a Pareto Type II distribution over $(0, \infty)$ (i.e., $1 - G_0(\alpha) = (1 + \alpha)^{-\beta}$), then $\alpha(F) = 0$ and hence while the centering is heavy-tailed, $F$ is super-heavy tailed. Motivated by this, in the next sections, we will emphasize scale mixtures of NGG processes in $\mathcal{N}$, that is,  \eqref{pymixtures}; the only exception is Section~\ref{simulation}, where we will consider once more \eqref{pymixtures2} for comparison purposes.

\section{Consequences and extensions}\label{sect:extensions}
\subsection{Multivariate variants}\label{multivariate}
We now discuss how Section~\ref{mixture} can be extended to define priors on the space of multivariate heavy-tailed distributions, that is, the class of joint distributions with heavy-tailed marginals. This yields the following multivariate versions of \eqref{pymixtures} and \eqref{pymixtures2} for $\mathbf{y} \in \mathbb{R}^d$.

First, extending \eqref{pymixtures} consider the following multivariate heavy-tailed {NGG} scale mixture model: 
\begin{equation}\label{pymixtures_multi} 
   \begin{cases} f(\mathbf{y}) = \int_{\mathbb{R}^d_+} \mathbf{K}(\mathbf{y}; \boldsymbol{\eta_\sigma})\, \dif G(\boldsymbol \sigma),\\ 
 G  \sim \text{NGG}(M, \tau, D, \red{G_0}), \\ (M, \tau, D) \in \mathcal{N}, \\
     {1 - G_{0, k}(\sigma) = \frac{\mathscr{L}_k(\sigma)}{\sigma^{\alpha_{0, k}}}}, \\
 \end{cases}
\end{equation}
\noindent where $\boldsymbol{\sigma} = (\sigma_1, \dots, \sigma_d) \in \mathbb{R}_+^d$, $\etab_{\sigmab} = (\eta_{\sigma_1}, \dots, \eta_{\sigma_d})$ and $G_{0,k}(\sigma)$ is the $k$th marginal distribution of $G_0(\boldsymbol{\sigma})$, for $k=1,\dots, d$. In terms of the kernel for the mixture model in \eqref{pymixtures_multi}, $\mathbf{K}(\mathbf{y}; \boldsymbol{\eta_\sigma})$, we assume that its margins are given by a scale kernel, $K_{\sigma_k}(y_k; \eta_{\sigma_k})$, that is 
\begin{equation}\label{kerns}
  \int_{\mathbb{R}^{\red{d - 1}}_+} \mathbf{K}(\mathbf{y}; \boldsymbol{\eta_\sigma}) \, \dif \mathbf{y}_{-k} =
  K_{\sigma_k}(y_k; \eta_{\sigma_k}), 
\end{equation}
with $\dif \mathbf{y}_{-k} = \dif y_{1}\dots\dif y_{k-1} \dif y_{k+1}\dots\dif y_{d}$. An example of a kernel obeying \eqref{kerns} is the conditional independence kernel, that assumes conditional independence among components along with a common parameter shared by all components, and which is given by $$\mathbf{K}(\mathbf{y}; \etab_{\sigmab}) = \prod_{k = 1}^d K_{\sigma_k}(y_k; \eta_{\sigma_k});$$
the latter kernel follows from the principles in \citet[][Section~3.1]{sarabia2008}. 

Second, extending \eqref{pymixtures2}, consider the following multivariate heavy-tailed NGG shape mixtures:
\begin{equation}\label{pymixtures2_multi}
  \begin{cases}
   f(\mathbf{y}) = \int_{\mathbb{R}^d_+} \mathbf{\mathcal{K}}(\mathbf{y}; \alphab; \boldsymbol{\eta_\alpha})\, \dif G(\boldsymbol \alpha), \\
   G \sim \text{NGG}(M, \tau, D, \red{G_0}), \\ (M, \tau, D) \in \mathcal{D},
  \end{cases}
\end{equation}

where $\boldsymbol{\alpha} = (\alpha_1, \dots, \alpha_d) \in \mathbb{R}_+^d$, and $G_{0,k}(\alpha)$ will denote the $k$th marginal distribution of $G_0(\boldsymbol{\alpha})$, for $k=1,\dots, d$. In terms of the kernel, for the mixture model in \eqref{pymixtures2_multi}, we assume that its margins are heavy-tailed, that is
\begin{equation}\label{kerns2}
  \int_{\mathbb{R}^{\red{d - 1}}_+} \mathbf{\mathcal{K}}(\mathbf{y}; \alphab, \boldsymbol{\eta_\alpha}) \, \dif \mathbf{y}_{-k} =
  \mathbb{k}(y_k; \alpha_k, \eta_{\sigma_k}), \\
  \quad 1 - \mathbb{K}(y; \alpha_k, \eta_{\alpha_k}) = \frac{\mathscr{L}_k(y)}{y^{\alpha_k}},
\end{equation}
where $\mathscr{L}_k$ is a sequence of slowly varying functions. An example of a kernel obeying \eqref{kerns2} is the following product of standard Pareto distributions $\prod_{k = 1}^d \alpha_k/y^{\alpha_k + 1}$, also constructed according to \citet[][Section~3.1]{sarabia2008}.

\begin{theorem}[Multivariate heavy-tailed NGG-mixtures]\label{props_multi} 
The following results hold { for $F_k$ the distribution function of the $k$th marginal distribution of $f$ in \eqref{pymixtures_multi} and \eqref{pymixtures2_multi}}:
  \begin{enumerate}[a)]
  \item  {If \eqref{pymixtures_multi} and \eqref{kerns} hold, 
    { with $U_k \sim K_{\sigma_k}(\cdot;\eta_{\sigma_k})$, $\E(U_{k+}^{\alpha_{0, k}}) < \infty$, $P(U_{k+} > \sigma) = o\{1 - G_0(\sigma)\}^{1/D}$ and  $\lim \inf_{\sigma \to \infty} \mathscr{L}(\sigma) > 0$, then the $k$th marginal of $f$ has a regularly varying tail with tail index $\alpha(F_k) = \alpha_{0, k} / D$, almost surely for $k=1,\dots,d$.}}
  \item If \eqref{pymixtures2_multi} and \eqref{kerns2} hold,   
    then {the $k$th marginal of $f$ has a regularly varying tail with tail index $\alpha(F_k) = \inf\{\alpha: G_{0, k}(\alpha) > 0\}$, for any $G \sim \text{NGG}(D, M, G_0(\alphab))$, almost surely.}
  \end{enumerate}
\end{theorem}

\noindent The previous result naturally extends Theorem~\ref{props} to the multivariate setting. 
\noindent 
\begin{remark}\label{remarkmodel} \normalfont Some comments on the multivariate heavy-tailed stable process scale mixtures in \eqref{pymixtures_multi} are in order:
\begin{itemize}
    \item A concrete instance of  \eqref{pymixtures_multi} and \eqref{kerns} that we have found to work well in practice is the following extension of Example~\ref{erlangex}, 
\begin{equation}\label{erlangmulti}
  f(\mathbf{y}) = \sum_{h=1}^{\infty} \pi_h \,\prod_{k=1}^{d} \Er(y_k; \lceil \sigma_{k,h} \rceil, \sigma_{k,h} / \lambda), 
\end{equation}
where once more $\Er(y; a, b)$ is the density of the univariate Erlang distribution and $\pi_h = V_h \prod_{k < h}(1 - V_k)$ with $V_h \overset {\text{ind}}{\sim} \Beta(1 - D,  h D)$, for $h \in \mathbb{N}$.
\item The specification in \eqref{pymixtures_multi} requires that the centering is itself a multivariate heavy-tailed distribution. These can be easily constructed from copulas \citep{nelsen2006}, and we recommend opting for Pareto Type II margins; that is, for $k = 1, \dots, d$,  
\begin{equation}\label{distributionPareto}
\begin{aligned}
1 - G_{0,k}(\sigma)&=
\left(1+\frac{\sigma}{\beta}\right)^{-\alpha_{0, k}}, 
\end{aligned}
\end{equation}
with $\sigma \geq 0$, $\alpha_{0, k} > 0$, and $\beta>0$. Such margins for the centering are convenient since they lead to a closed-form posterior of the extreme value index, as can be seen from the supplementary material (Section~3).
\item To complete \eqref{pymixtures_multi} we recommend a Jeffrey's prior on the tail index $\textsf{p}(\alpha_{0,k}) \propto 1/\alpha_{0,k}$ and a Beta prior on the discount parameter, $D\sim\mbox{Beta}(a_D,b_D)$. Finally, one may set a prior on the remainder parameters of the kernel, and for instance, in the Erlang kernel example in \eqref{erlangmulti}, we will opt for $\lambda\sim\text{Gamma}(a_{\lambda}, b_{\lambda})$. {This implies a prior for the tail index $\alpha_0/D$, with infinite expectation. Alternatively, with a proper prior on $\alpha_0$, e.g., $\alpha_0\sim\Gamma(a_{\alpha_0},b_{\alpha_0})$, the induced prior on the tail index would have prior expectation  $E(\alpha_0)E(1/D)=\{a_{\alpha_0}(a_D+b_D-1)\}/\{b_{\alpha_0}(a_D-1)\}.$}
\item The specifications in \eqref{pymixtures_multi} and \eqref{kerns}---as well as \eqref{pymixtures2_multi} and \eqref{kerns2}---can be used to model heavy-tailed partially exchangeable signals by considering a single parameter in the mixing. Given the prominence of partial exchangeability in Bayesian Nonparametric literature, we believe this to be an interesting feature of this specification \citep{camerlenghi2019b, camerlenghi2019a, catalano2021}.
\end{itemize}
\end{remark}

\subsection{Modeling conditional joint densities}\label{covariates}\label{conditional}

We now show how to extend the proposed models to include the effect of covariates 
by using a particular NGG process---the stable process. Specifically, a single-atoms dependent stable process is constructed following the principles from \citet[][Definition~3]{barrientos2012} and \citet[][Section~2.3]{quintana2022}. For conciseness, we focus on multivariate heavy-tailed NGG scale mixtures in \eqref{pymixtures_multi}, but the principles discussed below can be easily adapted for the multivariate shape mixtures from Section~\ref{multivariate} as well as to the univariate methods from Section~\ref{sect:the_model}. Consider the following predictor-dependent model,  
\begin{equation}\label{eq:modelcov}
f(\mathbf{y} \mid \mathbf{x})= \int_{\mathbb{R}^d_+}~\prod_{k=1}^{d} K_{\sigma_k}(y_k; \eta_{\sigma_k})\, \dif G_{\mathbf{x}}(\boldsymbol \sigma),
\end{equation}
where $\mathbf{y}, \boldsymbol \sigma \in \mathbb{R}^d$ and $\{G_{\mathbf{x}}\}$ is a family of random probability measures indexed by a covariate $\mathbf{x} \in \mathbb{R}^p$. Specifically, we consider the following dependent stable process 
 \begin{equation}\label{eq:discrete_measurecov}
   G_{\mathbf{x}} = \sum_{h=1}^{\infty} \pi_h({\mathbf{x}}) \delta_{\boldsymbol{\sigma}_h},
   \qquad \boldsymbol\sigma_h\iid G_0(\boldsymbol\sigma). 
 \end{equation}
Here, the weights of the stick-breaking representation and the discount parameter $D$ of the NGG process are indexed over the covariate as follows, $\pi_{h}(\mathbf{x}) =V_h(\mathbf{x})\prod_{k<h}\{1-V_k(\mathbf{x})\}$, and
\begin{equation}\label{vds}
V_h(\mathbf{x})\sim\text{Beta}(1-D_h(\mathbf{x}),hD_h(\mathbf{x})), \qquad 
D_h(\mathbf{x}) = \frac{e^{\mathbf{x}^{\T}\boldsymbol\beta_h}}{1+ e^{\mathbf{x}^{\T}\boldsymbol\beta_h}},
\end{equation}
where $\boldsymbol\beta_h$ is a parameter in $\mathbb{R}^p$. Clearly, since \eqref{eq:discrete_measurecov} is a NGG for every $\mathbf{x}$, Theorem~\ref{props_multi}~a) implies that the joint density mixture model in \eqref{eq:modelcov} yields a  multivariate  heavy-tailed distribution, for every $\mathbf{x}$. The model is completed with a prior distribution for $\boldsymbol\beta_h$, given by $\boldsymbol{\beta}_h \iid N_p(\boldsymbol{0},s^2 \boldsymbol{I})$, for $h \in \mathbb{N}$.
A specific embodiment of the approach discussed in this section, that will be revisited later in the paper, is the following extension of Example~\ref{erlangex}, 

\begin{equation}\label{erlangmulticov}
    f(\mathbf{y} \mid \mathbf{x}) = \sum_{h=1}^{\infty} \pi_h(\mathbf{x}) \,\prod_{k=1}^{d} \Er(y_k; \lceil \sigma_{k,h} \rceil, \sigma_{k,h} / \lambda).
  \end{equation}

\section{Simulation study}\label{simulation}
\subsection{Simulation scenarios and preparations}\label{oneshot}
This section describes the true data-generating processes and the settings used over the Monte Carlo simulation study from Section~\ref{mc}.  \vspace{0.2cm}

\noindent \textbf{Data-generating processes}: We consider one scenario for the univariate version of the model from Section~\ref{mixture}, three scenarios for the multivariate version from Section~\ref{multivariate}, and three scenarios for the multivariate conditional version from Section~\ref{conditional}. The univariate scenario is a standard unit Pareto distribution with tail function $1-F(y)=1 / y$, for $y > 1$, and its main aim will be to highlight that for heavy-tailed data, NGG-$\mathcal{N}$ mixing leads to much better fits at the tails than Dirichlet process mixing. Beyond the univariate scenario, we also considered bivariate and conditional scenarios that contemplate different dependence levels and complexities of the marginals. Table~\ref{table:sim} summarizes the bivariate and conditional scenarios, which are marginally characterized by 
\begin{equation}\label{eq:marginals_simulation}
  \begin{cases}
    f_k(y)=w\,f_{\text{LG}}\{y\mid a_1, b_1\}+(1-w)\,f_{\text{LG}}\{y\mid a_2,b_2\}, \\
    f_k(y\,|\, x)=w\,f_{\text{LG}}\{y\,|\, a_1(x),b_1(x)\}+(1-w)\,f_{\text{LG}}\{y\,|\, a_2(x),b_2(x)\}, 
  \end{cases}
\end{equation}
for $k=1,2$, where $f_{\text{LG}}(y; a, b)=b^a\log(y)^{a-1}y^{-(b+1)}/\Gamma(a)$ is the density of a log-gamma distribution with shape $a>0$ and rate $b>0$; parenthetically, we note that the log-gamma distribution is in the Fr\'{e}chet domain of attraction with tail index $b$ \citep[][Table 2.1]{beirlant2004}.

The dependence is modeled via a Gumbel copula, so that data for the bivariate and conditional scenarios are respectively simulated from  
\begin{equation}\label{eq:copula_simulation}
  \begin{cases}
    F(y_1, y_2)=C_\theta\{F_1(y_1),F_2(y_2)\}, \\
    F(y_1, y_2\mid x)=C_\theta\{F_1(y_1\mid x),F_2(y_2\mid x)\}.
\end{cases}
\end{equation}
Here, $C_{\theta}(u, v) = \exp[-\{(-\log u)^\theta + (-\log v)^\theta\}^{1/\theta}]$, for $(u, v) \in (0, 1)^2$, whereas $\theta \geq 1$ is the parameter controlling dependence, and $F_1$ and $F_2$ are the distribution functions of $f_1$ and $f_2$. For all scenarios,  we have simulated $n = 1\,000$ observations, and for the conditional scenarios, covariates were drawn from a standard uniform distribution.

\begin{table}[H]
\centering
\caption{Bivariate and Conditional Simulation Scenarios. The marginals are mixtures of log-gamma distributions as in \eqref{eq:marginals_simulation}, and dependence is set by a Gumbel copula with parameter $\theta$.}
\begin{center}
\begin{tabular}{lllc}
    \hline \hline 
    \textbf{Scenario} & \textbf{Marginal} $(f_{1})$  & \textbf{Marginal} $(f_{2})$  & \textbf{Copula} $(\theta)$ \\\hline 
    ~~~Bivariate~1 & $a_1=a_2=5;w=1$   & $a_1=a_2=5;w=1$   & 3 \\ \hline 
    \phantom{~~~Bivariate}~2 & $a_1=a_2=5;w=1$   & $a_1=a_2=5;w=1$   & 1   \\ \hline 
  \phantom{~~~Bivariate}~3 & $a_1=13;b_1=7;$   & $a_1=8;b_1=7;$  &  \multirow{2}{*}{1}\\ 
               &  $a_2=10; b_2=8;w=.4$  & $a_2=15; b_2=8;w=.4$  &  \\ \hline \hline 
  Conditional~1   & $a_1=1+4x;a_2=3;w=1$    & $a_1=1+4x;a_2=3;w=1$    & 1 \\ \hline 
  \phantom{Conditional}~2   & $a_1=1+4x;a_2=3;w=1$    & $a_1=1+4x;a_2=3;w=1$    & 3\\ \hline 
  \phantom{Conditional}~3  & $a_1=11+5x;b_1=8+5x;$ & $a_1=6+5x;b_1=12+5x;$   &  \multirow{2}{*}{1}\\ 
                              & $a_2=b_2=7; w=.4$  &  $a_2=b_2=8;w=.4$& \\ \hline \hline    
\end{tabular}
\end{center}
\label{table:sim}
\end{table}

\noindent \textbf{MCMC and model specification}: 
All models were fitted using {an efficient slice sampler} \citep{kalli2011} available from the supplementary material (Section~3). We considered a burn-in period of 5\,000 iterations, and after that,  scanned 5\,000 samples from the posterior targets of interest.
For the univariate scenario, we fitted a particular NGG-$\mathcal{N}$ mixture; that is, we fit a stable process scale mixture with an uninformative gamma prior and an Erlang kernel (i.e., \eqref{erlang2} with prior $\lambda \sim \Ga(0.1, 0.1)$); 
for the bivariate and conditional scenarios we fitted stable process scale mixtures with an Erlang kernel based on Remark~\ref{remarkmodel} and Equation~\eqref{erlangmulticov}, respectively. For the latter, a Gumbel copula was used for the centering, and empirical Bayes was used to set the hyperparameter for $\theta$ via maximum likelihood. For the conditional version of the model in the regression parameters in \eqref{vds}, we consider the prior $\boldsymbol{\beta}_h \iid N_p(\boldsymbol{0},s^2 \boldsymbol{I})$, and set the hyperparameter to be $s^2=100$. For the hyperparameters of the marginals Pareto for the base measure, we set $\beta_k=1$ and tail index $\alpha_{0, k} = 2$, for $k = 1, 2$, which implies that both margins are apriori heavy-tailed but have a finite expected value. Finally, we have assigned the prior  $D \sim \Beta(0.5, 0.5)$ to the discount parameter for all instances of the model. Keeping in mind space constraints and the preference for stable process scale mixtures noted in Section~\ref{mixture}, here we mainly concentrate on assessing the performance of the latter. The posterior inference algorithms available from the supplementary material (Section~3) can, however,  be used for fitting multivariate as well as conditional heavy-tailed NGG shape mixtures, and some instances of the latter are available from the \texttt{NGGR}~package. \vspace{0.2cm}

\noindent \textbf{One-shot experiments}: 
One-shot experiments for the bivariate and conditional scenarios are presented in the supplementary material (Section 3). All in all, the resulting fits suggest that the proposed methods accurately recover the true distribution for all scenarios being examined. Such findings should, of course, be regarded as tentative, as they are the outcome of a single-run experiment and will be subject to the scrutiny of the Monte Carlo simulation study in the next section.

\begin{figure}
\begin{center}
 \includegraphics[scale=0.45]{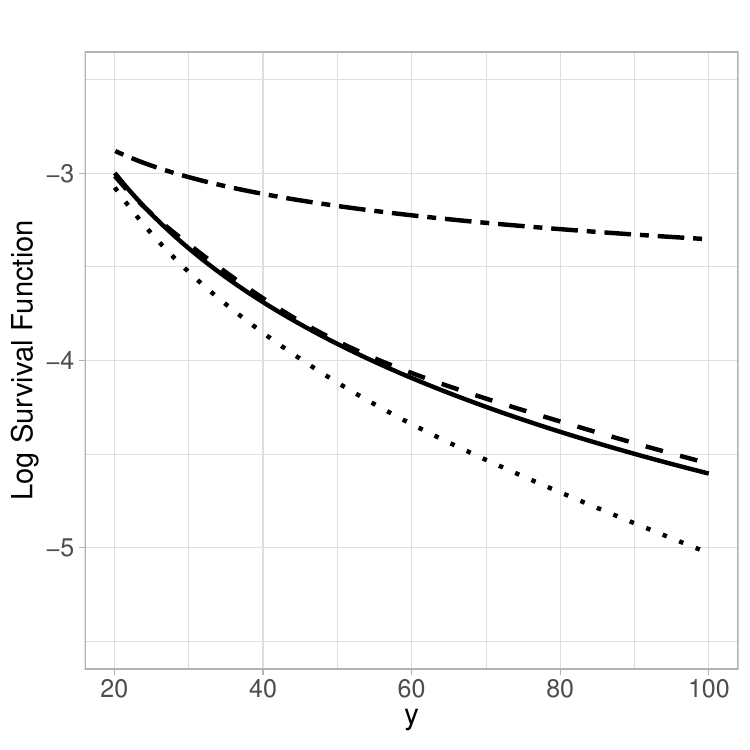} \footnotesize
\caption{\label{fig:mcmcuniv}
  Mean of the Monte Carlo fits (dashed line) for the log-survival function for the  univariate scenario obtained using a particular NGG-$\mathcal{N}$ mixture model (i.e., stable process scale mixture) from Section~\ref{mixture}, focusing from the $95\%$ to the $99\%$ quantile, plotted against the true (solid line). The dotted line shows the Monte Carlo mean of the fits from an NGG-$\mathcal{D}$ mixture with the same Erlang kernel, whereas the dashed--dotted line shows the Monte Carlo mean of the fits from an NGG-$\mathcal{D}$ mixture with a Pareto kernel and a gamma centering distribution.}
\end{center}
\end{figure}

\subsection{Monte carlo simulation study}\label{mc} 
We now report the main findings of a Monte Carlo simulation study. For each scenario from  Section~\ref{oneshot}, we simulated 100 data sets,  each containing $n = 1\,000$ observations. All models have been fitted using stable process scale mixture models with the same specifications and MCMC settings described in Section~\ref{oneshot}. We start with the univariate unit Pareto scenario, which will reinforce our preference for NGG-$\mathcal{N}$ scale mixture models.

In Figure~\ref{fig:mcmcuniv},  we present the posterior Monte Carlo means of the log-survival estimates for the tail of the univariate scenario and compare it with the corresponding Monte Carlo mean for a Dirichlet process mixture based on the same kernel. As can be seen from Figure~\ref{fig:mcmcuniv}, a stable process mixture accurately estimates the tail, whereas a Dirichlet process mixture markedly underestimates it; this numerical evidence showcases that the proposed stable process scale mixtures are a natural option for modeling risk and extremes in a heavy-tailed framework. Such numerical performance of the proposed methods finding is not surprising in light of Theorem~\ref{props}~a); the performance over the bulk (not shown) is comparable for both forms of mixing. Interestingly, Figure~\ref{fig:mcmcuniv} also reveals that the Monte Carlo mean based on fitting the Dirichlet process mixture with a Pareto kernel and a gamma-centering distribution overestimates the tail of the distribution. Such numerical finding is not surprising, keeping in mind Theorem~\ref{props}~b), given that the left endpoint of the gamma centering distribution function is 0 and hence the resulting mixture is super heavy-tailed (i.e., $\alpha(F) = 0$). Next, we move to the bivariate and conditional scenarios from Section~\ref{oneshot}. Figure \ref{fig:mcmc} shows 100 posterior estimated contours for the three bivariate scenarios. As can be seen from the latter figure, the proposed stable process scale mixture model can capture the true contours over different levels of dependence (Scenarios 1--2) and even with challenging marginals such as mixtures (Scenario~3). The results for the conditional scenarios are presented in the supplementary material (Section 4) and also suggest an overall good performance of the proposed methods.

\begin{figure}
\begin{center}
\hspace{-.8cm}\includegraphics[scale=0.35]{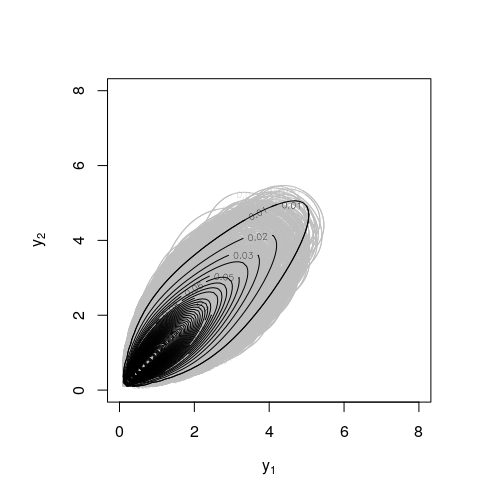}\hspace{-.6cm}\includegraphics[scale=0.35]{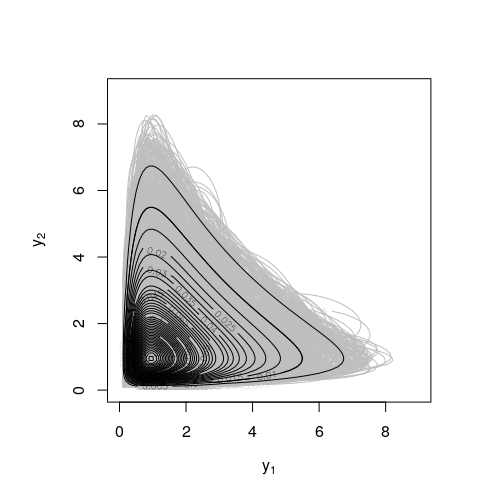}\hspace{-.6cm}\includegraphics[scale=0.35]{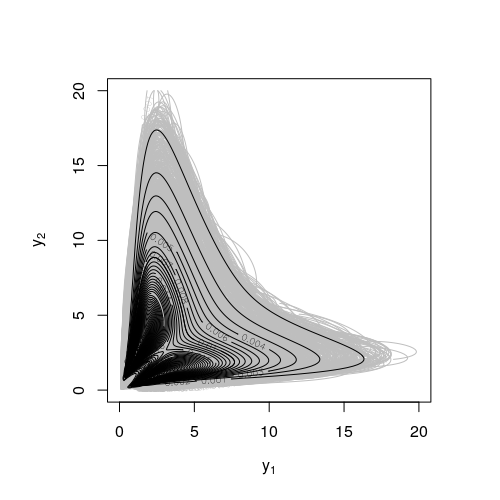}\vspace{-0.4cm}
\caption{ \label{fig:mcmc} Monte Carlo Simulation for the Bivariate Scenarios 1--3: Contours of the joint density estimates (gray) obtained with proposed stable process scale mixture model from Section~\ref{multivariate}, for the 100 simulated data sets, plotted against the true (black).}
\end{center}
\end{figure}

\section{Application to heavy-tailed brain data}\label{application}
\subsection{Applied context and data description}\label{data}

We now showcase the application of the proposed methods to a neuroscience case study. Brain rhythm signals are key for understanding how the human brain works; loosely speaking, they consist of patterns of neuronal activity that are believed to be linked with certain behaviors, arousal intensity, and sleep states \citep{Frank2009}. These signals are typically measured using an electroencephalogram (EEG), which records electrical activity in the brain via electrodes attached to the scalp. An EEG signal tracks the activity of billions of neurons, and such signals cover a broad spectrum of frequency bands. Say, the alpha band typically refers to 8--13Hz, while beta refers to 13--20Hz; for a primer on brain rhythms and EEG signals, see, for instance, \cite{buzsaki2006} and \citet[][Chapter~7]{ombao2016}.

Alpha and beta rhythms are believed to be heavy-tailed \citep[e.g.,][]{roberts2015}. Hence, the main goal of our analysis will be to learn about the marginal and joint distribution of these heavy-tailed oscillations, given a variety of stimuli to be described below. 
In the supplementary material (Section~5), we report evidence supporting the claim that in line with \cite{roberts2015} our alpha and beta brainwave data are indeed heavy-tailed. We assess this by learning from data about the so-called extreme value index of a generalized Pareto distribution---which is known to be positive for heavy-tailed data \citep[][Section~4]{coles2001}. The data to be analyzed are available from the R package~\texttt{NGGR}, and were gathered from a UC Berkeley study that involved 30 participants who were subject to several audio-visual activities and stimuli, namely: {\it mathematics}, {\it relaxation}, {\it music}, {\it color}, {\it video} as well as  {\it relax and think}.

\begin{figure}[H]
{\centering \footnotesize \textbf{Mathematics} \hspace{5.3cm} \textbf{Relaxation}}  \\
\hspace*{-.5cm} \includegraphics[scale=0.268]{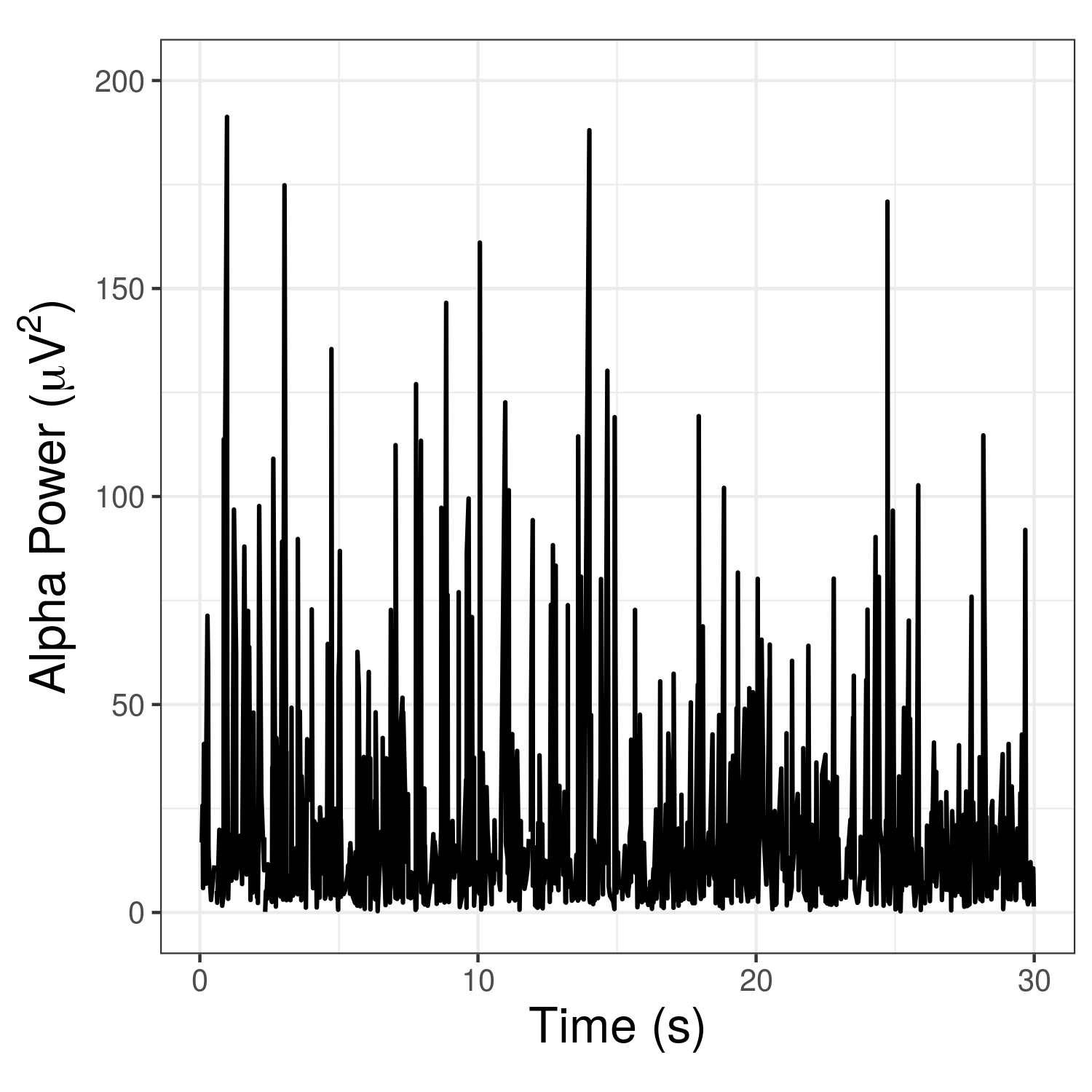}\includegraphics[scale=0.268]{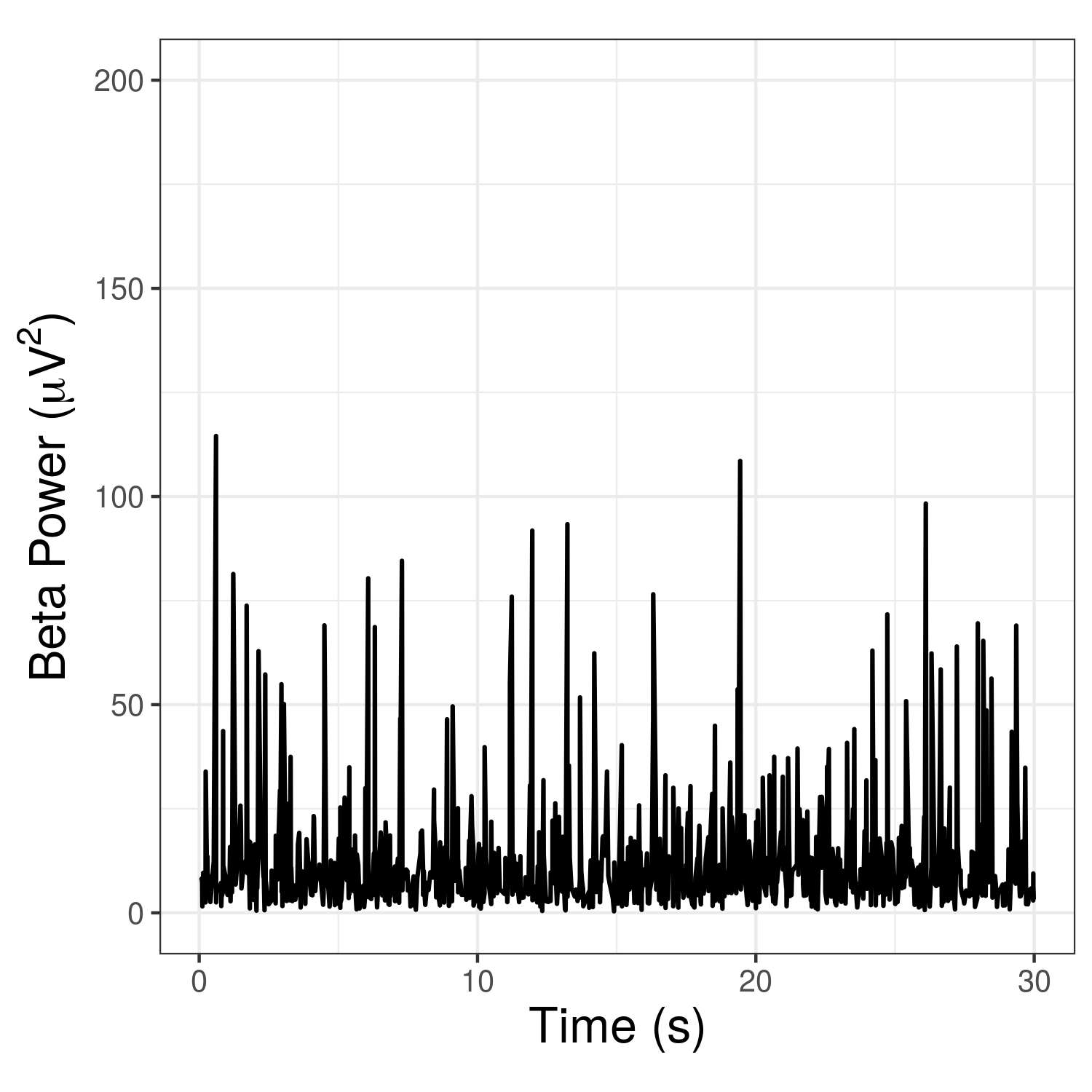}
\includegraphics[scale=0.268]{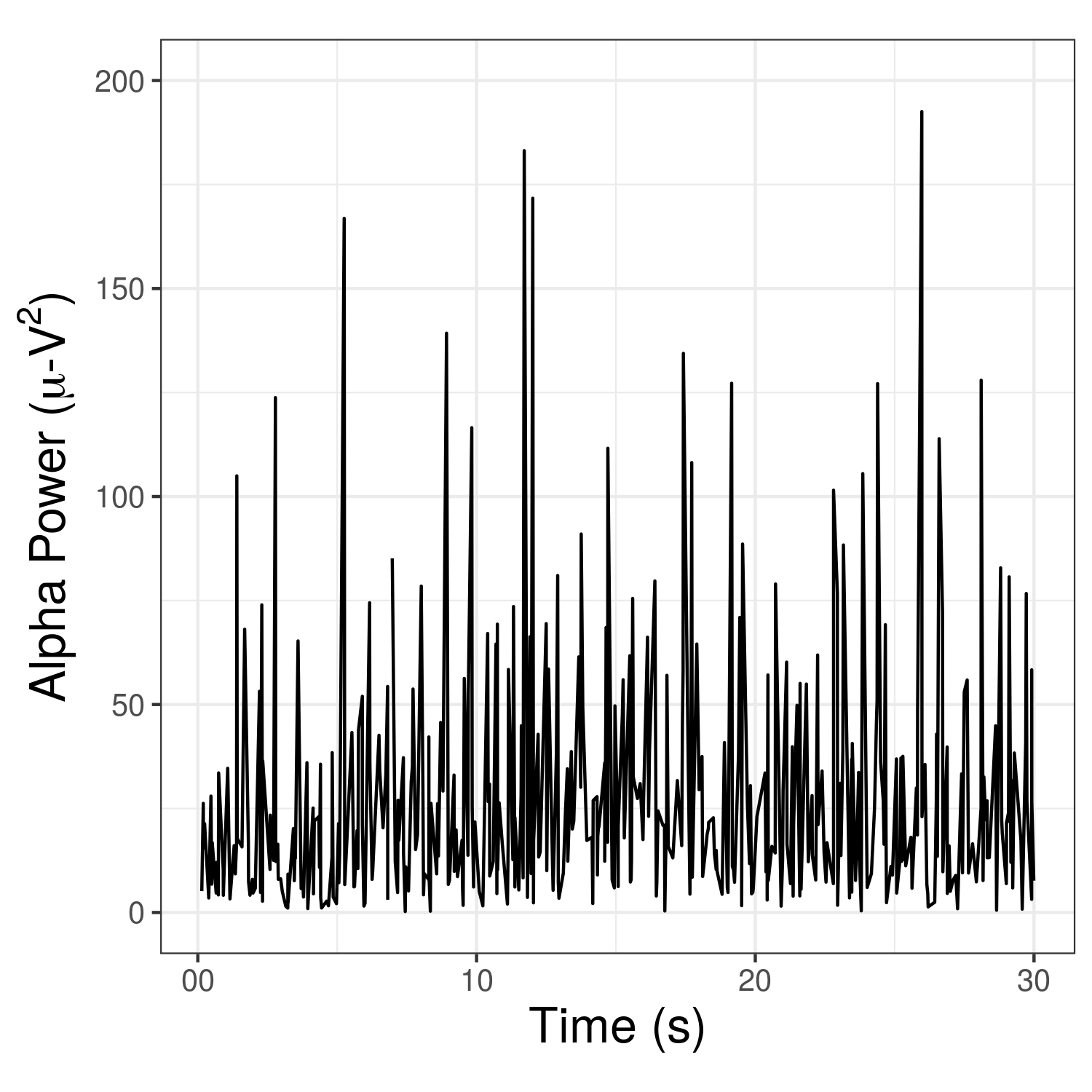}\includegraphics[scale=0.268]{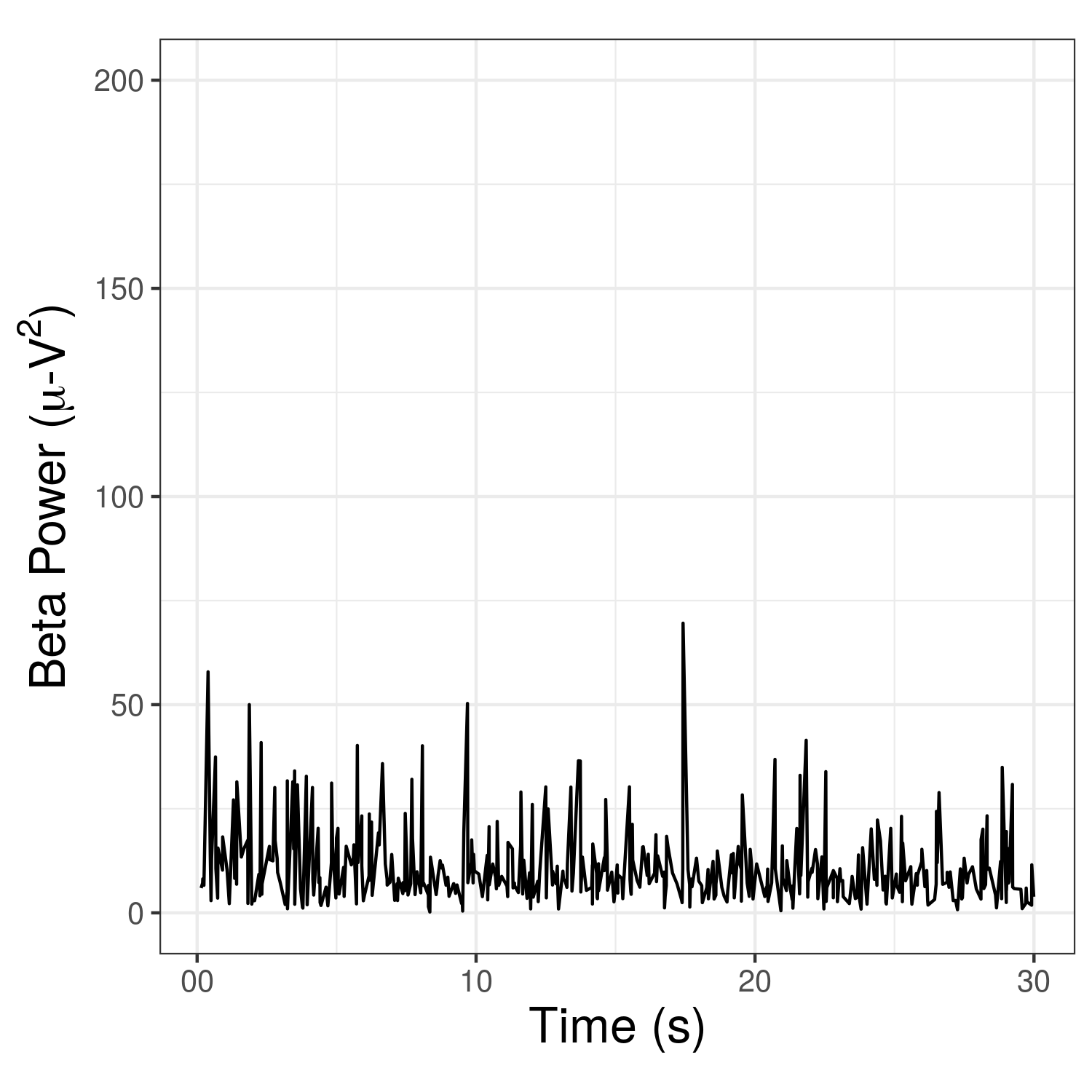} \\ 
{ \centering  \footnotesize  \textbf{Music} \hspace{6cm} \textbf{ Color}} \\
\hspace*{-.5cm}\includegraphics[scale=0.268]{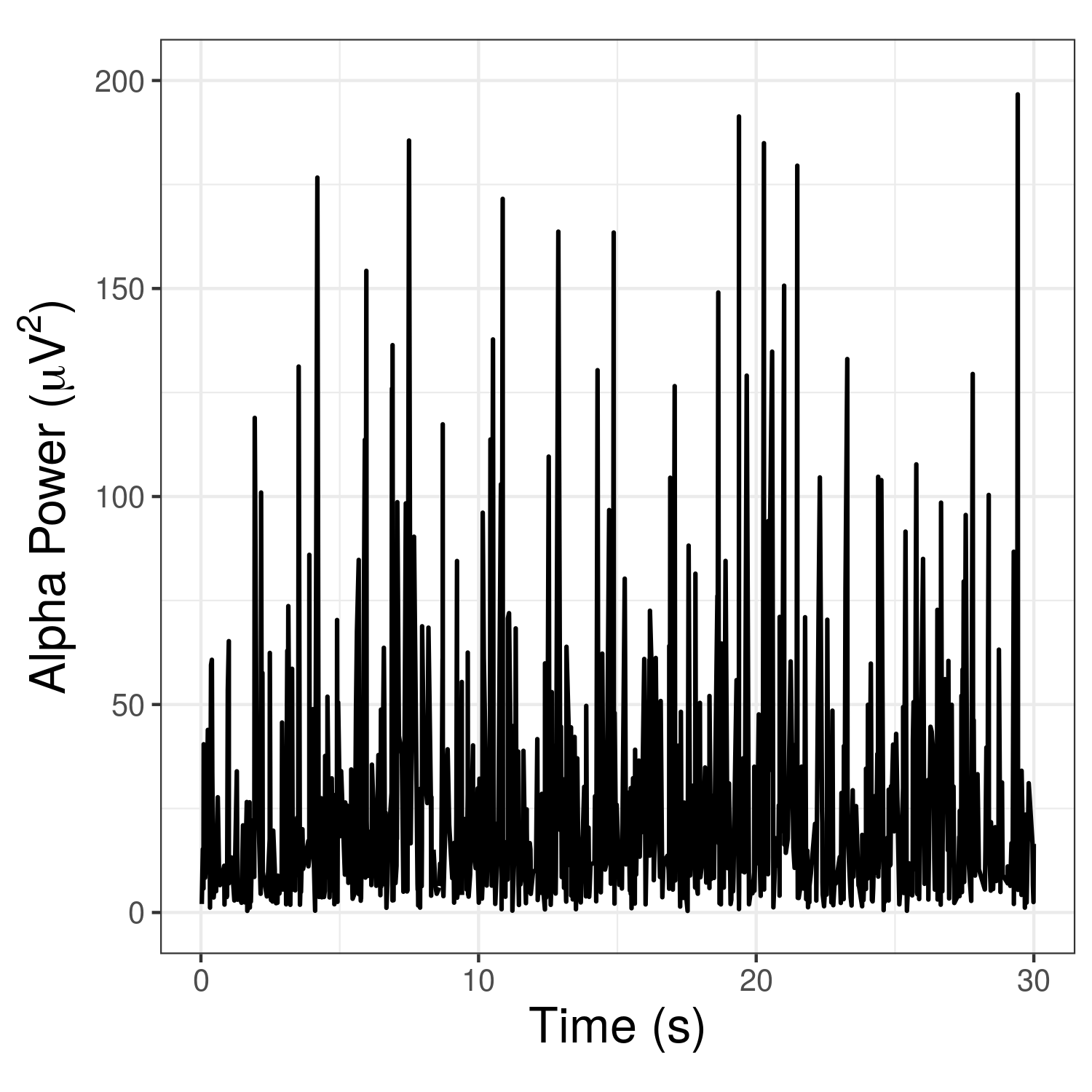}\includegraphics[scale=0.268]{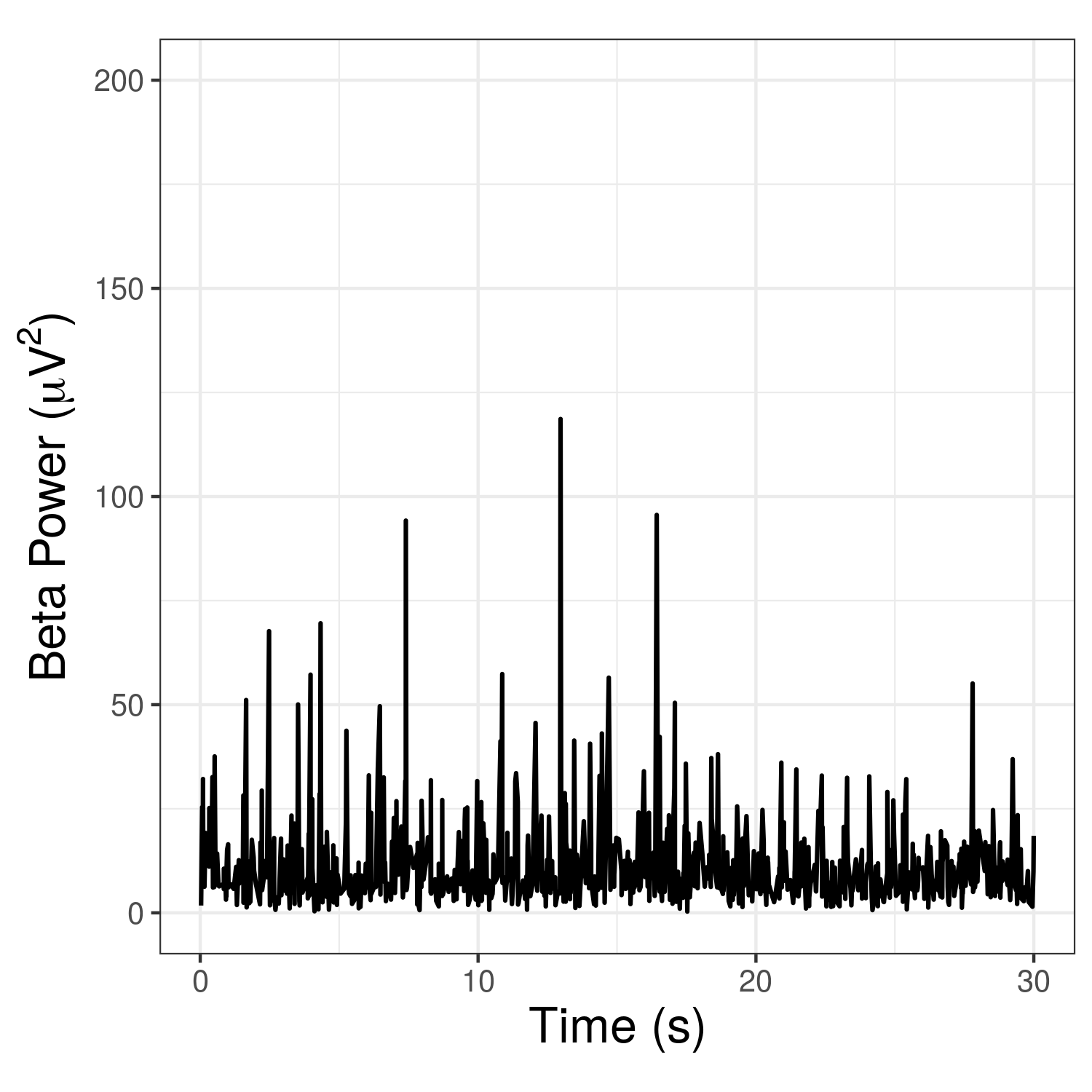}
\includegraphics[scale=0.268]{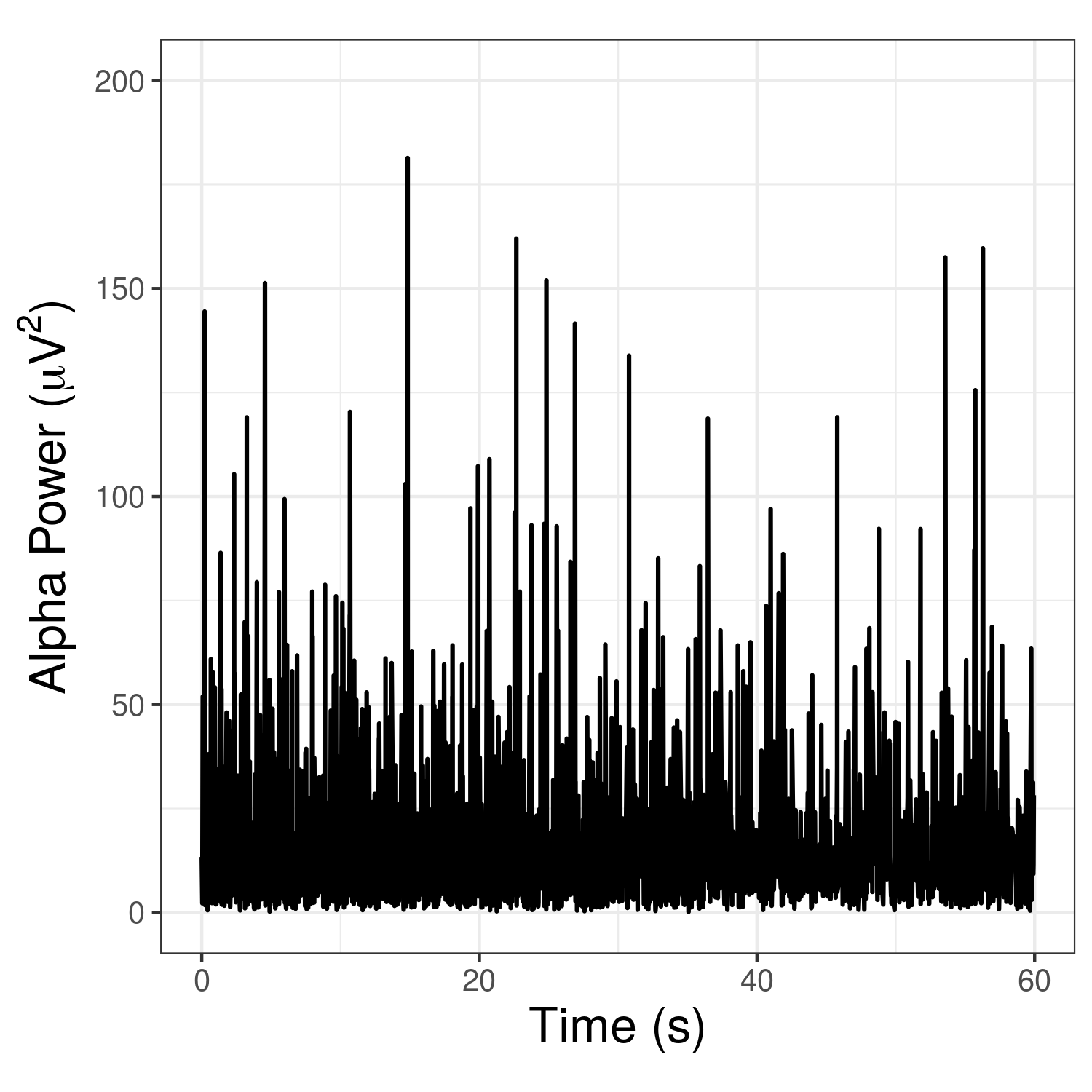}\includegraphics[scale=0.268]{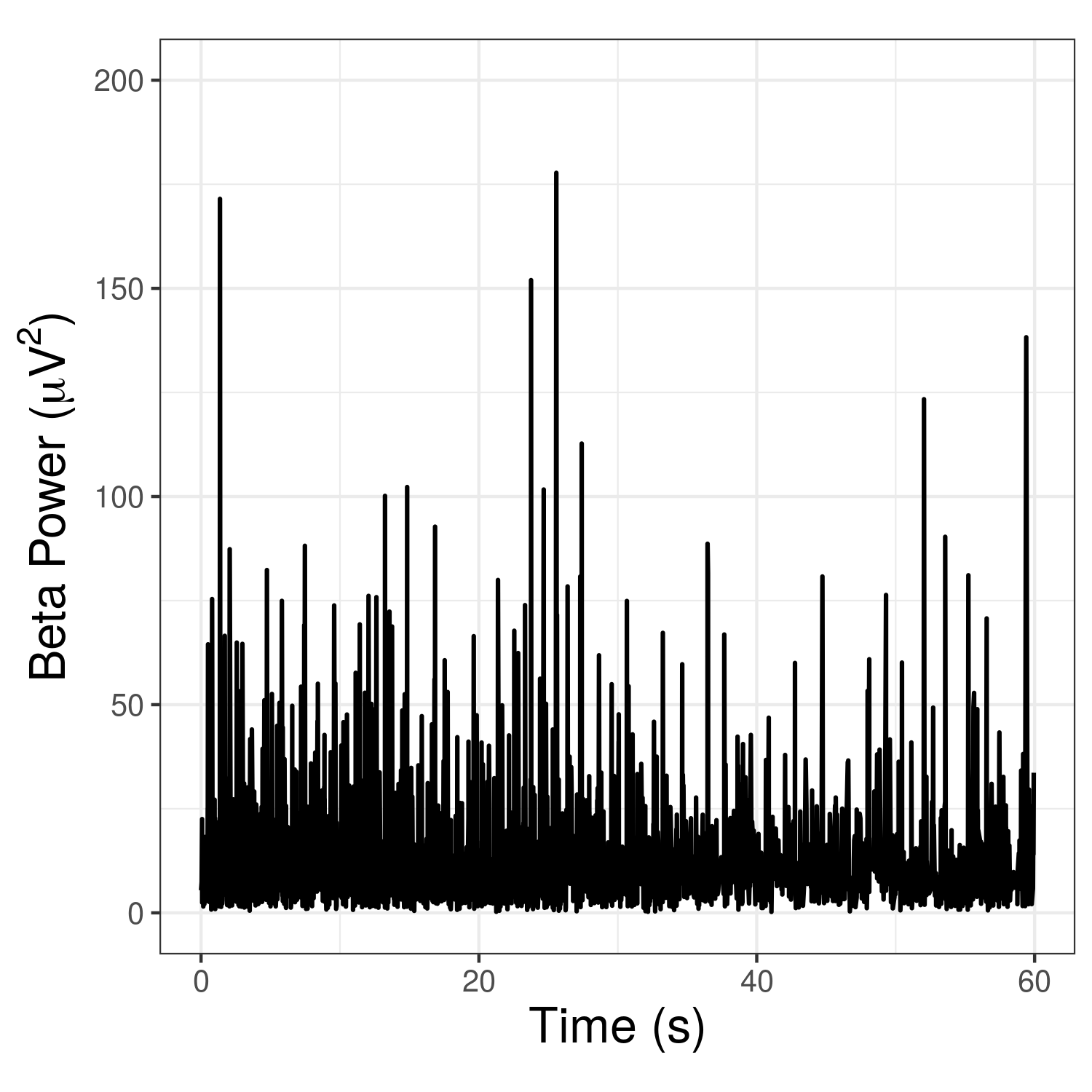} \\
{\centering \footnotesize \hspace*{.4cm} \textbf{Video}\hspace{5.6cm} \textbf{Relax and think}}  \\
\hspace*{-.5cm}\includegraphics[scale=0.268]{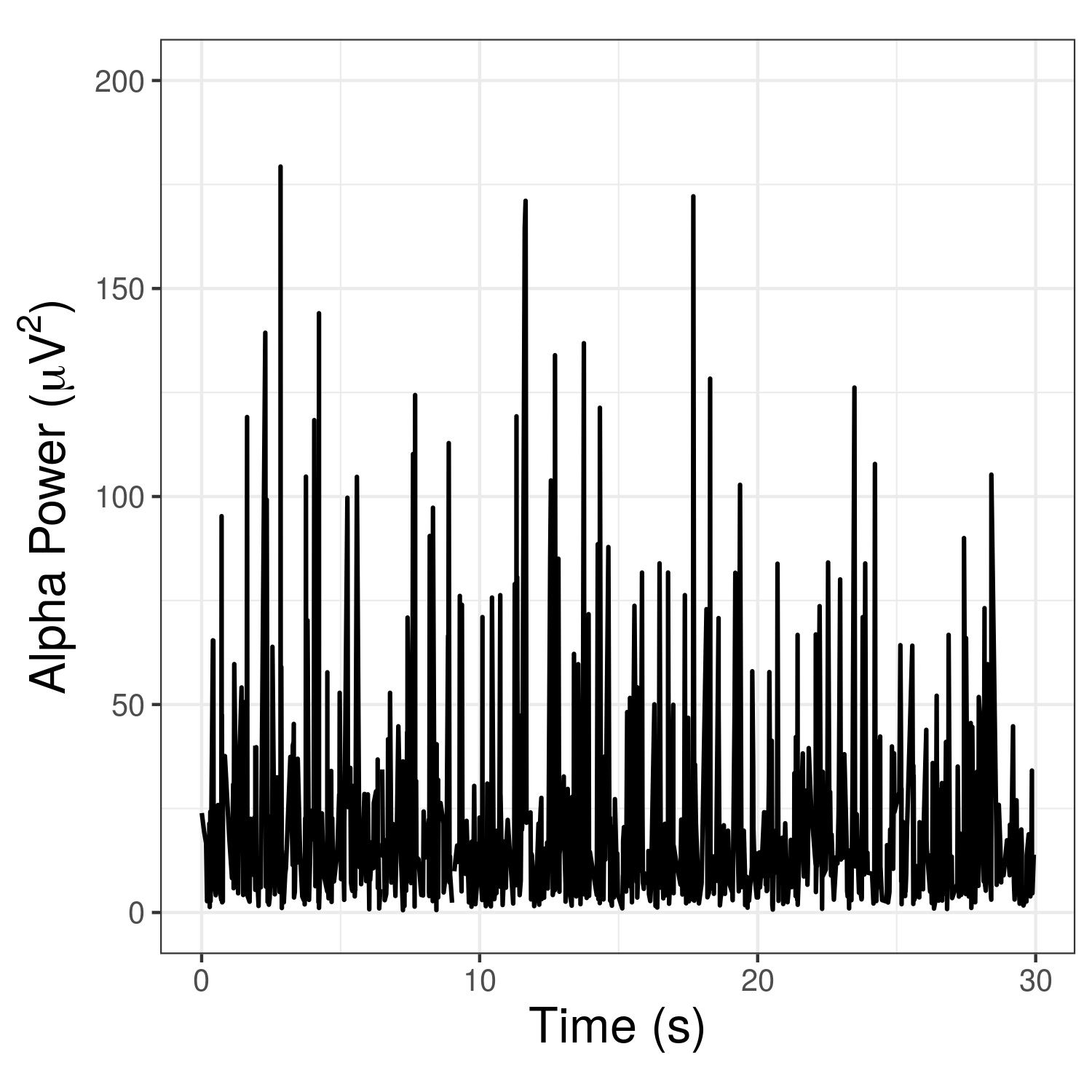}\includegraphics[scale=0.268]{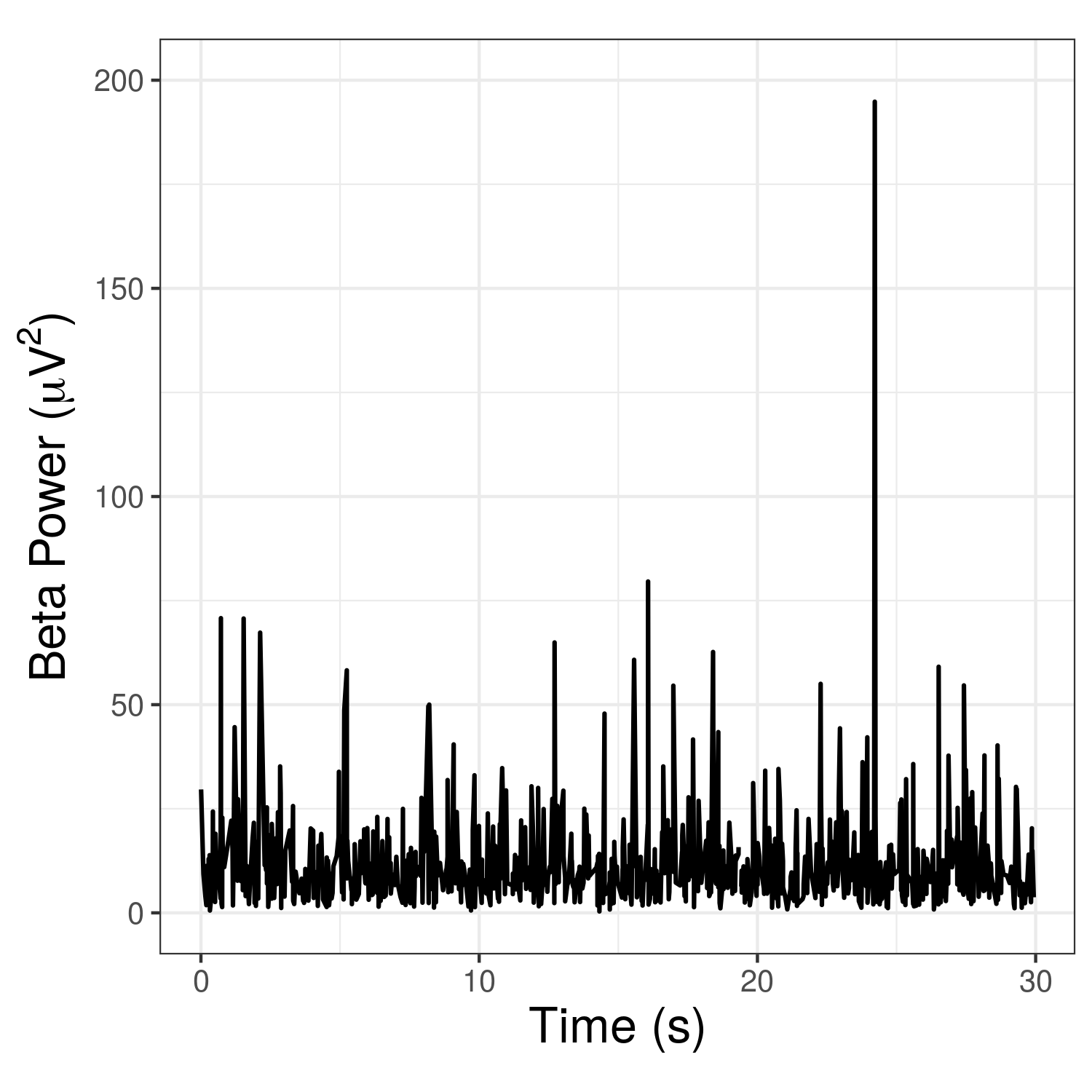}
\includegraphics[scale=0.268]{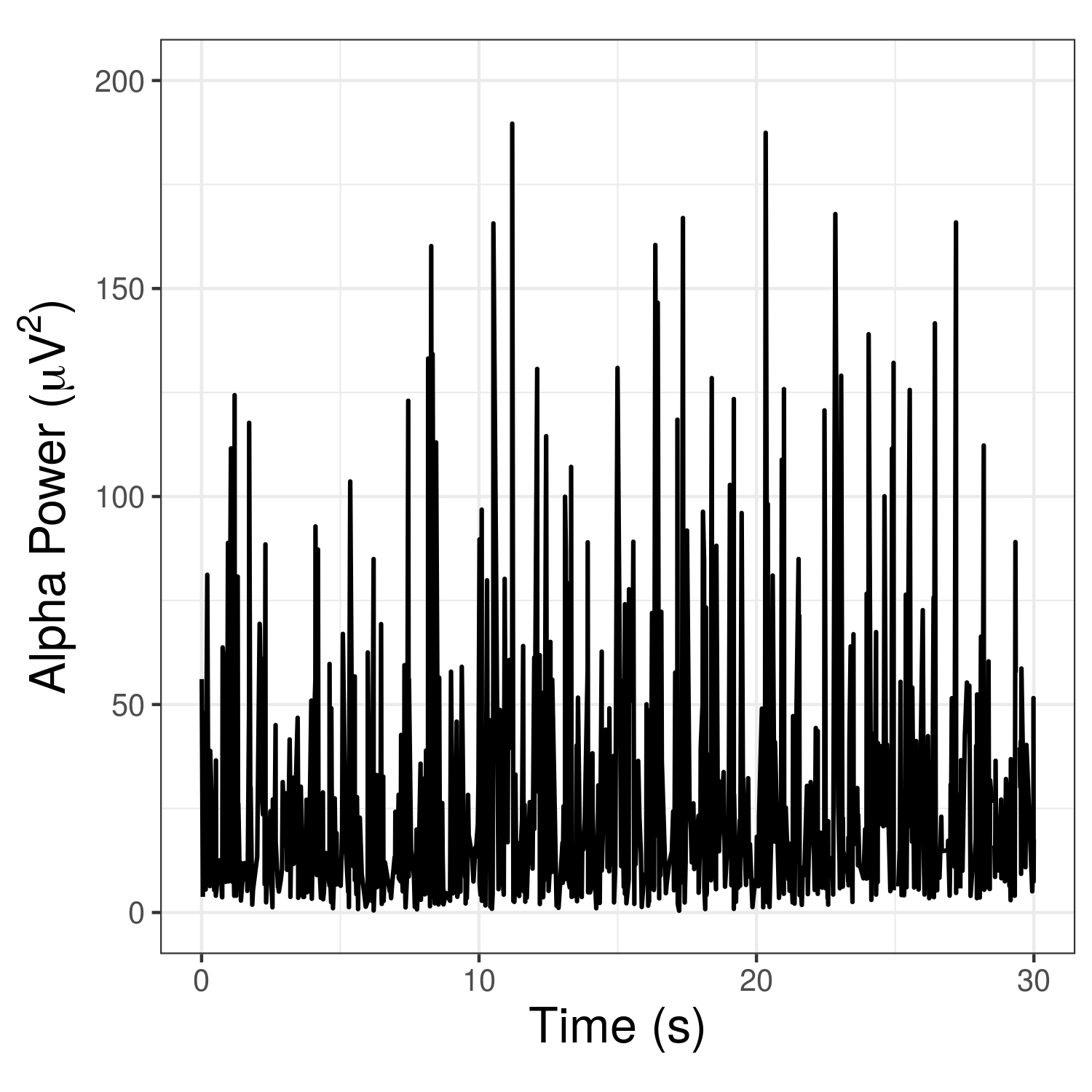}\includegraphics[scale=0.268]{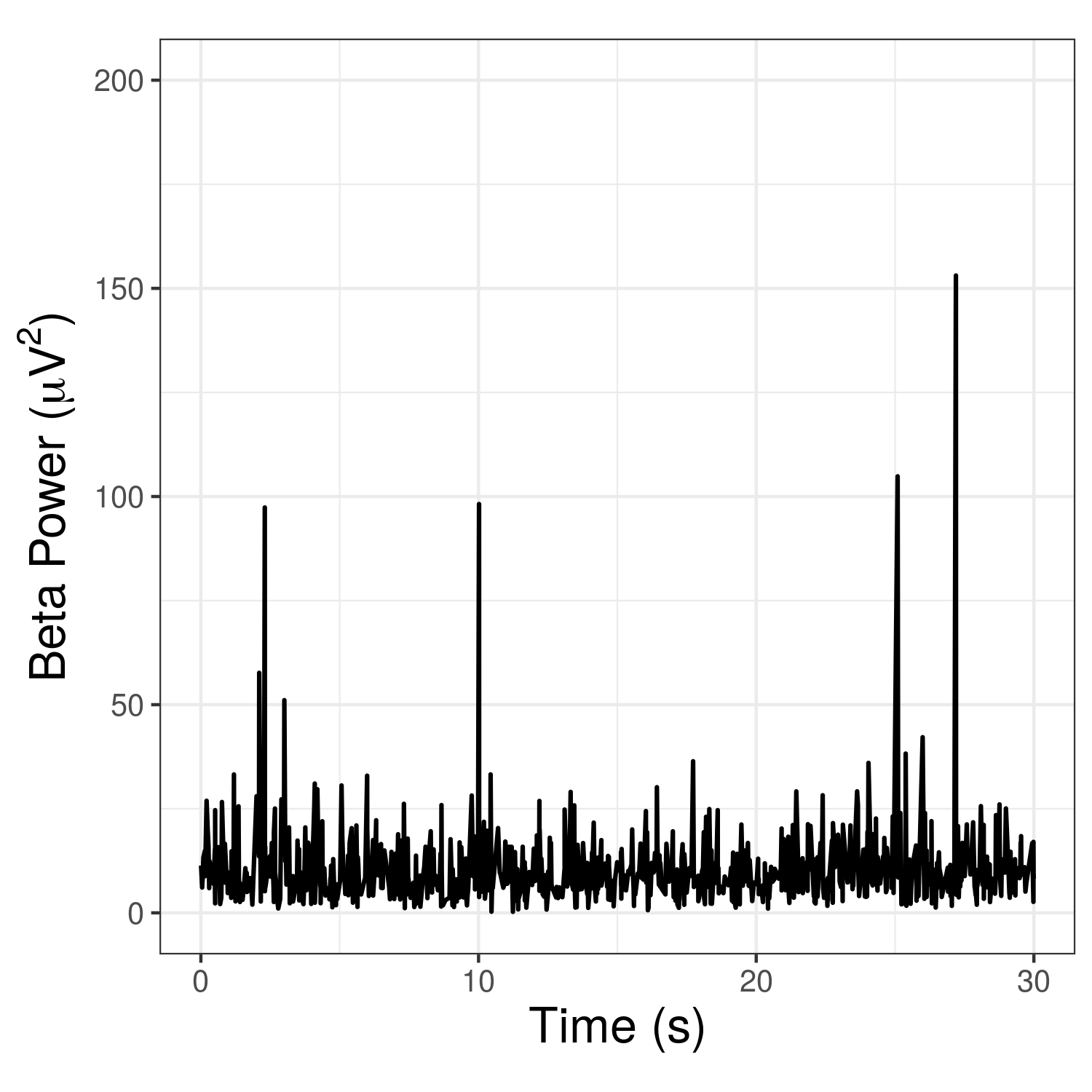} \vspace{-0.2cm}
\caption{ \label{fig:power} Spectral power ($\mu V^2$) of alpha and beta brainwave data plotted against time in seconds ($s$).}
\end{figure} 
Figure~\ref{fig:power} shows the spectral power in (micro)-Volts squared ($\mu V^2$) for alpha and beta waves for all  participants; roughly speaking, higher peaks indicate higher neural activity at a certain point in time on a frequency band of interest. The Ljung--Box test results reported in the supplementary material (Section~5) suggest that the recorded trajectories of spectral power can be regarded as independent over time across all stimuli. Figure~\ref{fig:power} illustrates some aspects of alpha and beta bands that help to build intuition on their signatures for the different stimuli. For instance, when the stimulus is \textit{mathematics} we can notice high activity for both alpha and beta waves, similar behavior to watching a \textit{video} or finding the \textit{color}---all these being tasks that relate to  immediate attention. Indeed, it has been suggested that alpha bands tend to be associated with `attention' as well as with `information processing' \citep{klimesch2012}. Interestingly, the patterns of alpha and beta waves for both \textit{mathematics} and \textit{music} are reasonably similar---which might not be surprising in light of what has been claimed elsewhere \citep[e.g.,][and references therein]{boettcher1994}. Whether that similarity of \textit{mathematics} and \textit{music} also holds for the \textit{joint} distribution of alpha and beta waves is something to be examined below in Section~\ref{brainwave}. Next, we learn from the heavy-tailed brainwave data discussed above using the methods proposed in Sections~\ref{sect:the_model}--\ref{sect:extensions}. Except where mentioned otherwise, all fits have been conducted using the same model specifications and MCMC settings as in Section~\ref{simulation}. 

\subsection{Marginal brainwave analysis}\label{mardata} 
Figure~\ref{fig:margins} shows the marginal density estimates of alpha and beta power pooling all subjects for each stimulus that were obtained using the proposed stable process scale mixture model from Section~\ref{mixture}. Specifically, the fits from Figure~\ref{fig:margins} were obtained using the specification in \eqref{erlang2} along with an uninformative gamma prior and an Erlang kernel. To assess the quality of the obtained fits, we depict in Figure~\ref{fig:margins} q-q boxplots \citep{rodu2022} of random quantile residuals \citep{dunn1996}. 

\afterpage{
  \renewcommand{\thefigure}{5}
\begin{figure}[H]
  \thispagestyle{empty}
  \vspace{-1.2cm}
{\centering \scriptsize \hspace*{0.15cm} \textbf{Mathematics} \hspace{5.2cm} \textbf{Relaxation}}\\
\centering \includegraphics[scale=0.26]{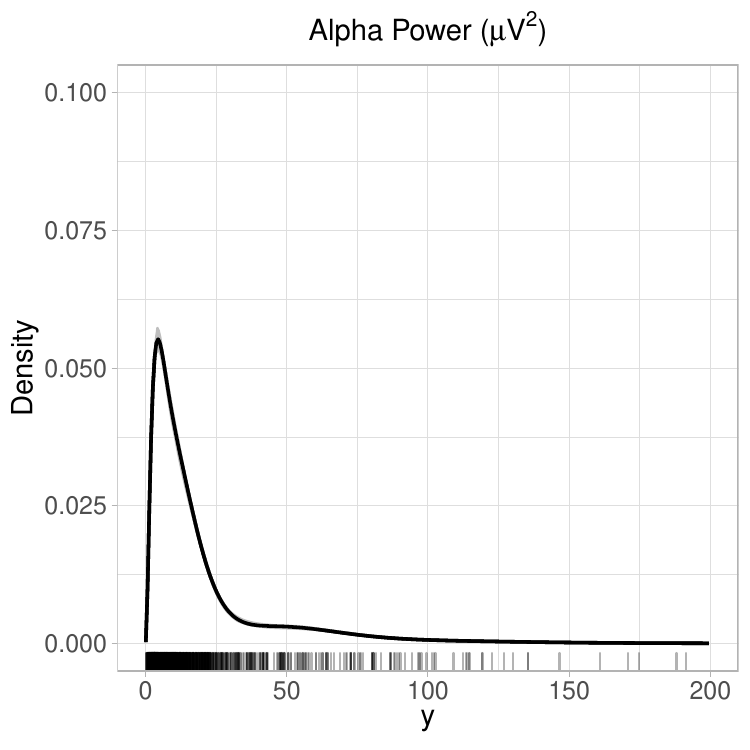}\includegraphics[scale=0.26]{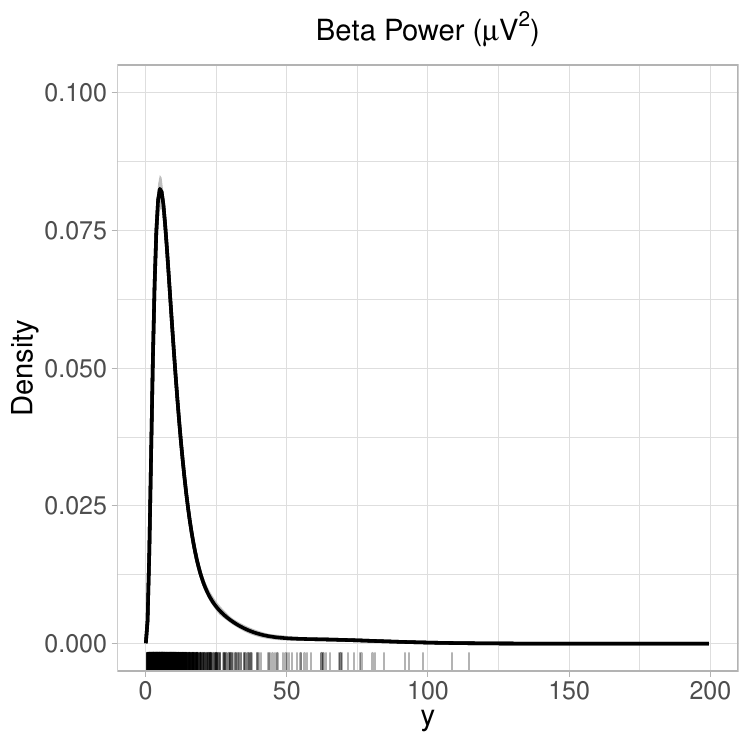}
\centering \includegraphics[scale=0.26]{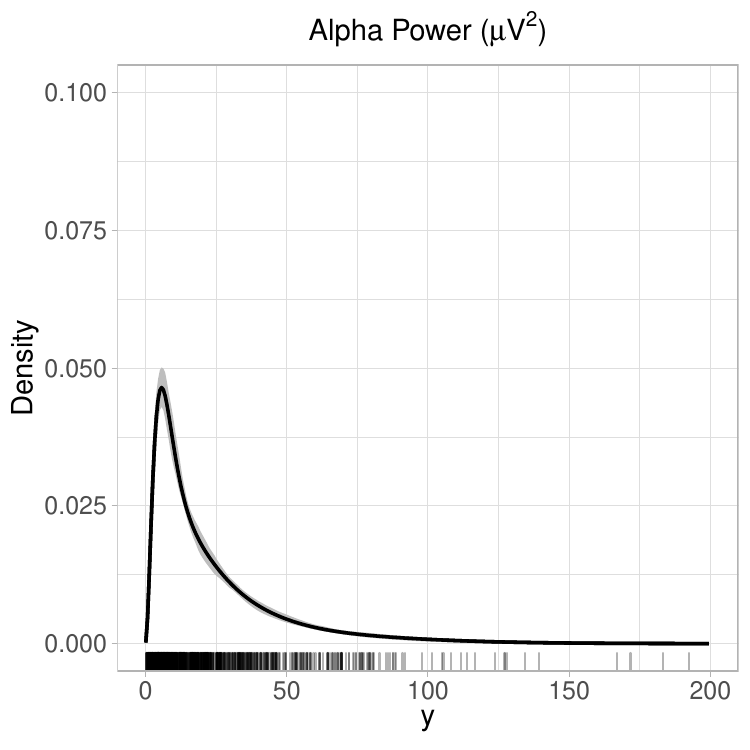}\includegraphics[scale=0.26]{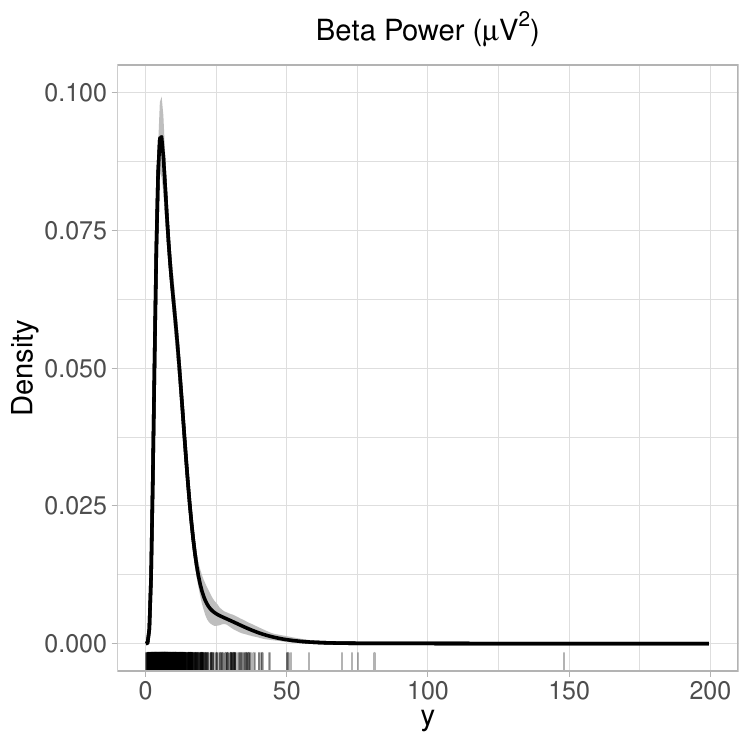} \\ \vspace{-0.4cm}
{\centering \scriptsize \hspace*{.3cm} \textbf{Music} \hspace{5.8cm} \textbf{ Color}} \\
\centering \includegraphics[scale=0.26]{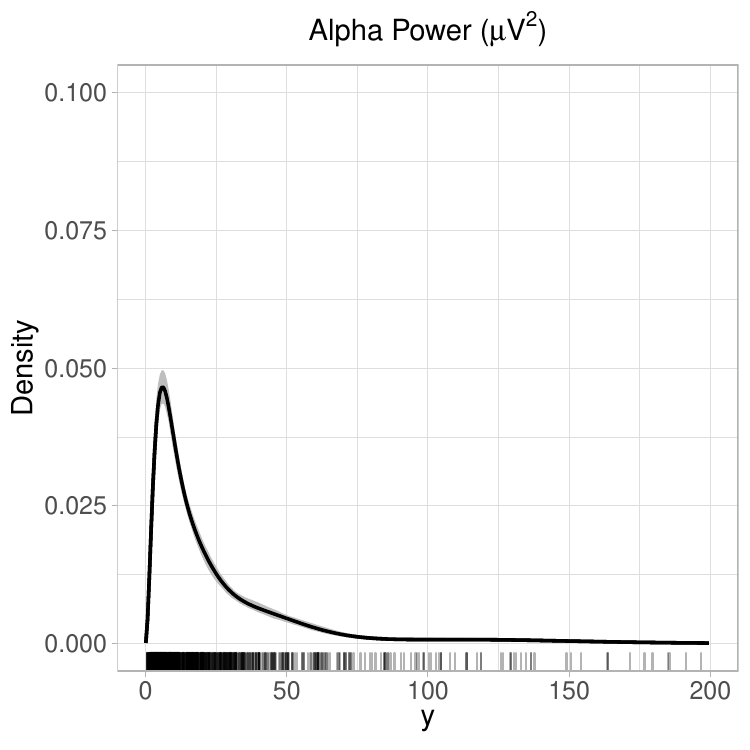}\includegraphics[scale=0.26]{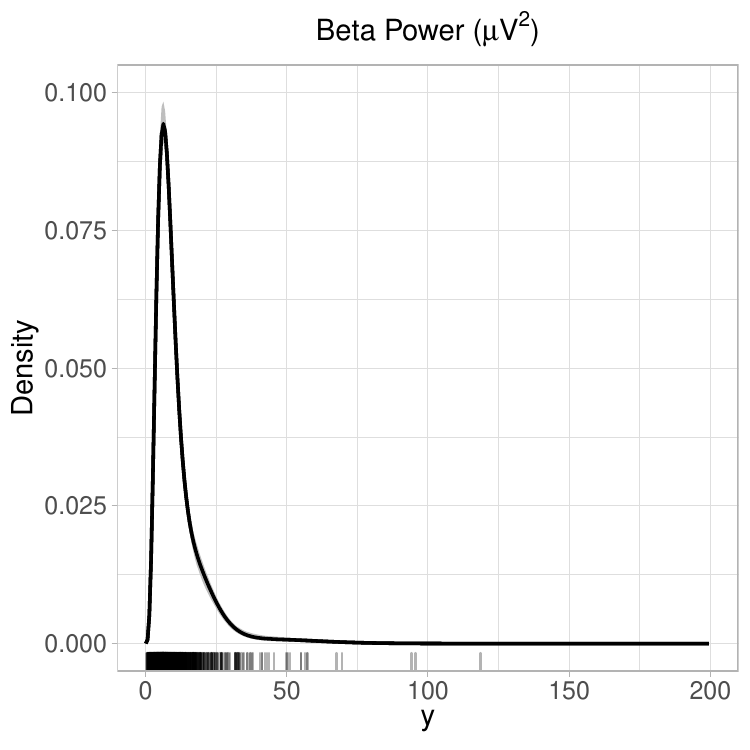}
\centering \includegraphics[scale=0.26]{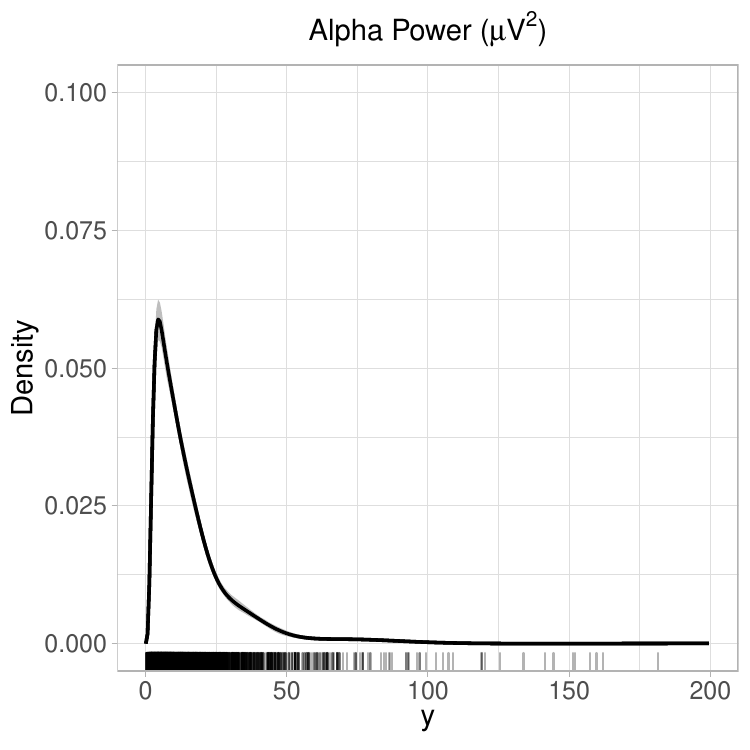}\includegraphics[scale=0.26]{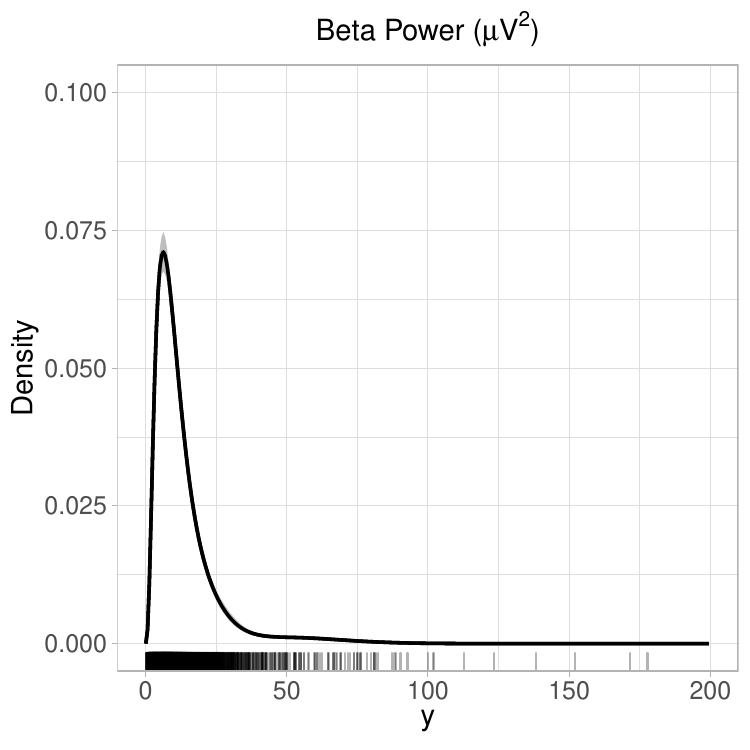} \\ \vspace{-0.4cm}
{\scriptsize \hspace*{.8cm} \textbf{Video}\hspace{5.3cm} \textbf{Relax and think}} \\
\centering \includegraphics[scale=0.26]{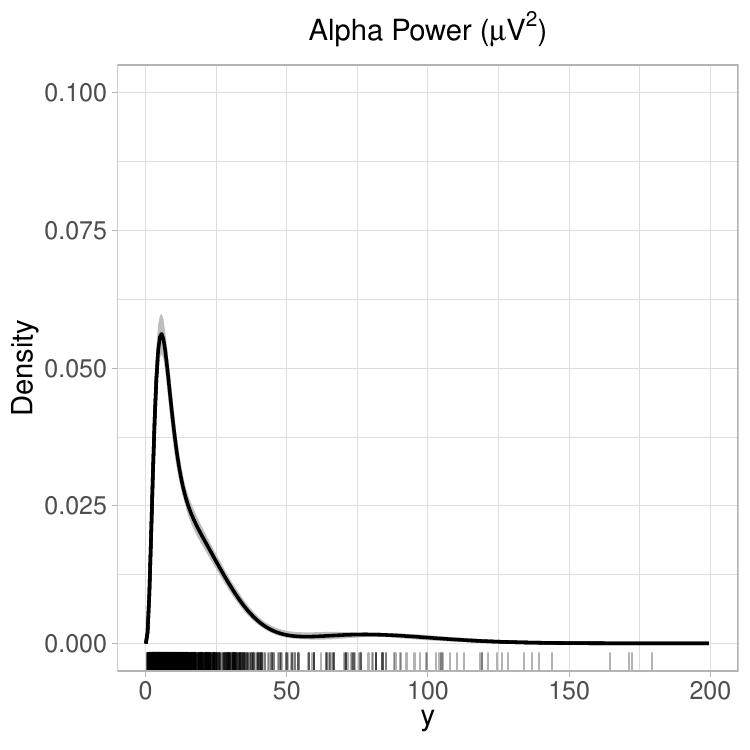}\includegraphics[scale=0.26]{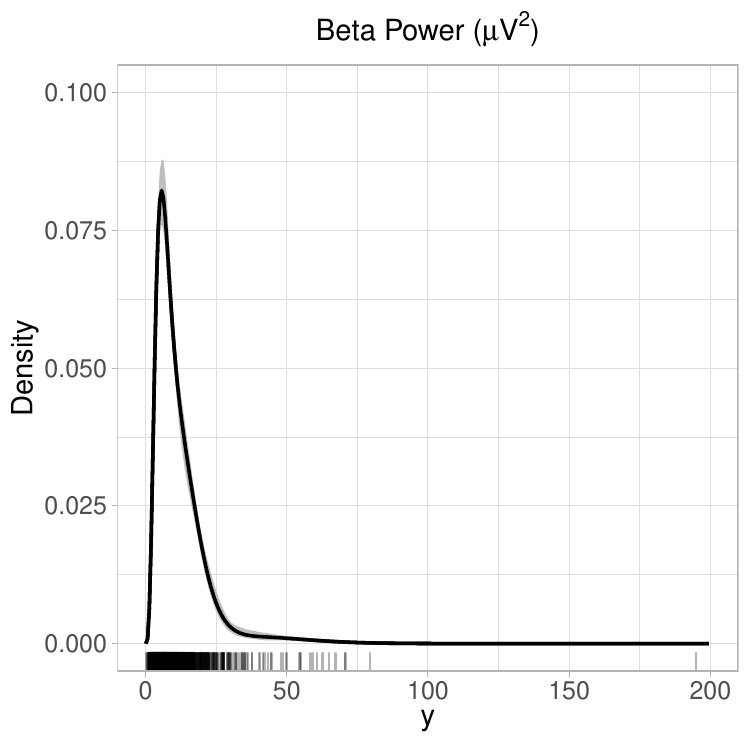}
\centering \includegraphics[scale=0.26]{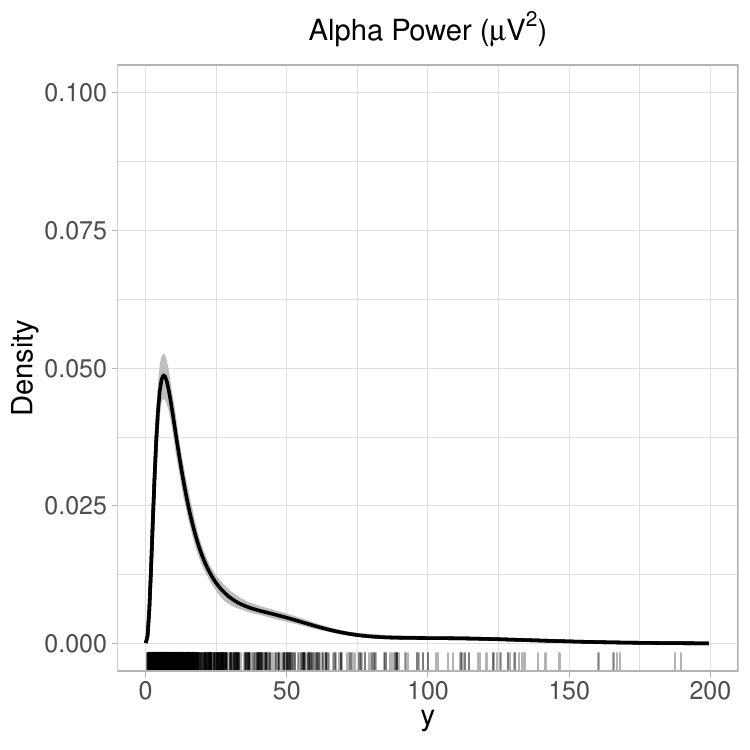}\includegraphics[scale=0.26]{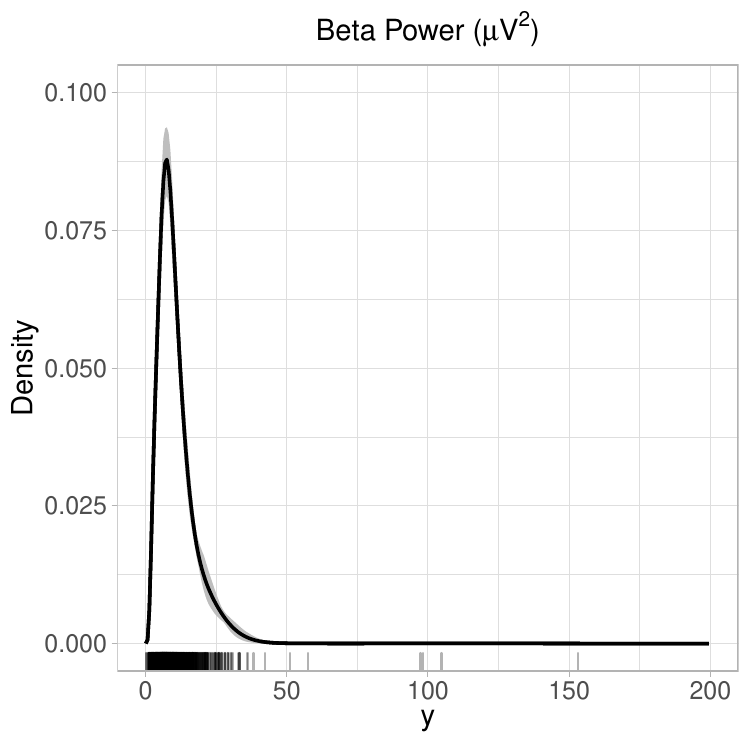} \\ \vspace{-0.4cm}
{\centering \scriptsize \hspace*{0.15cm} \textbf{Mathematics} \hspace{5.2cm} \textbf{Relaxation}}\\
\centering \includegraphics[scale=0.26]{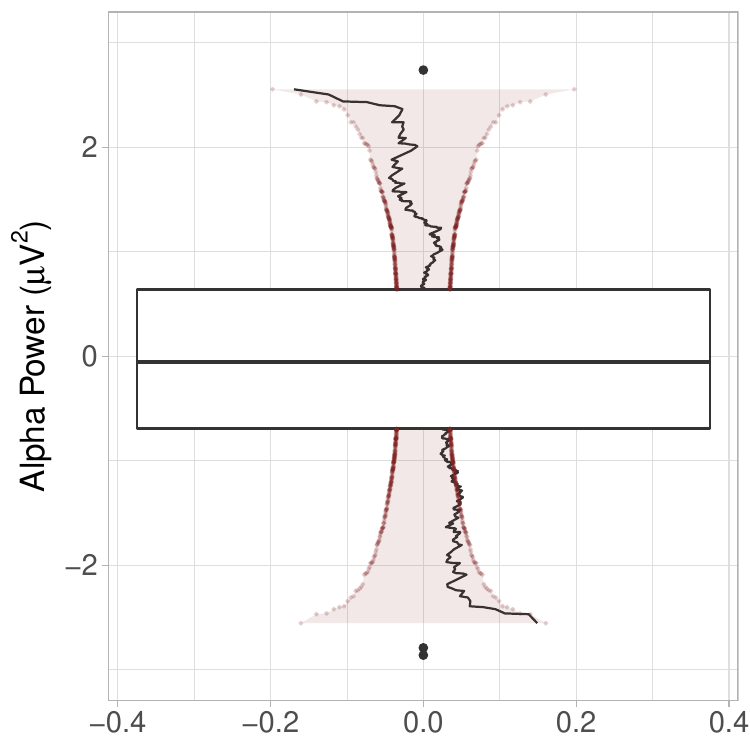}\includegraphics[scale=0.26]{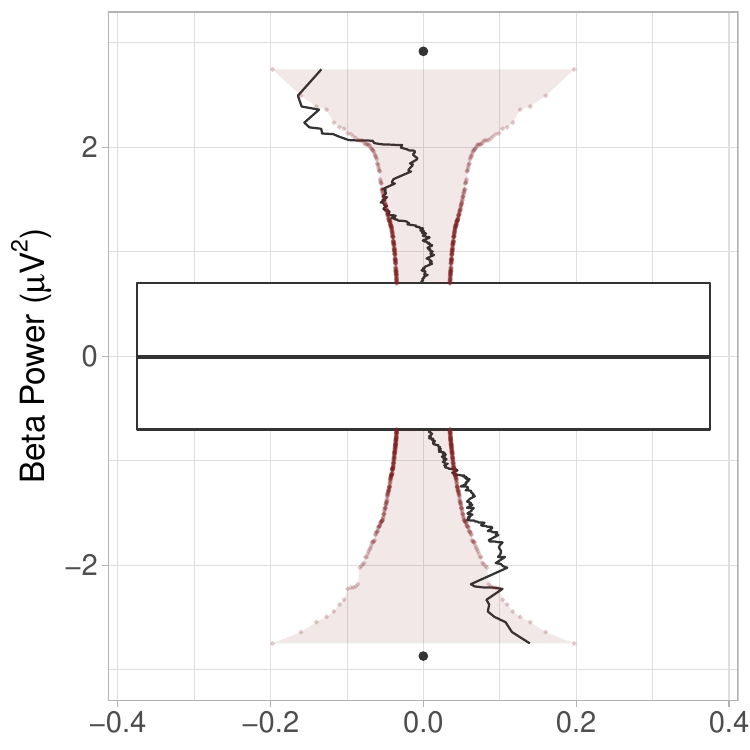}
\centering \includegraphics[scale=0.26]{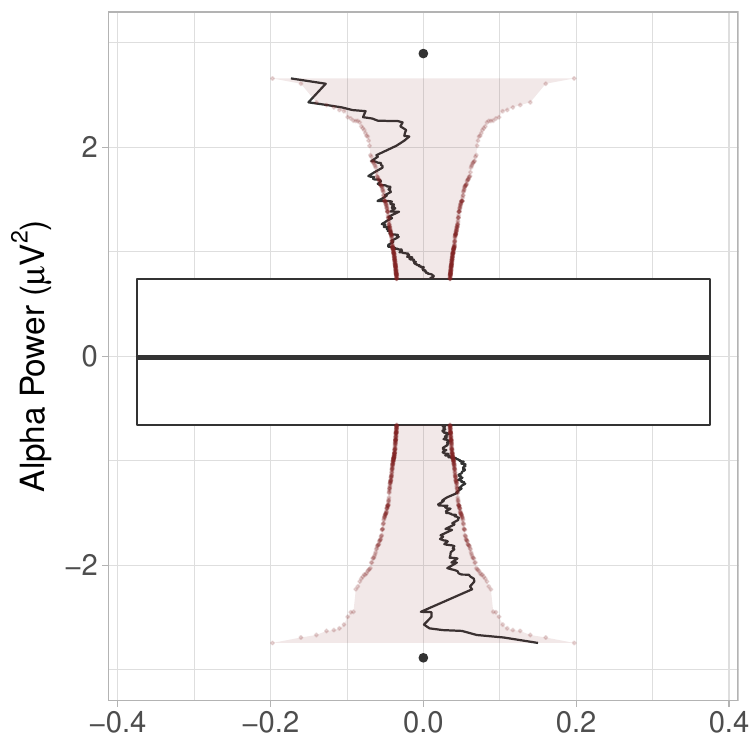}\includegraphics[scale=0.26]{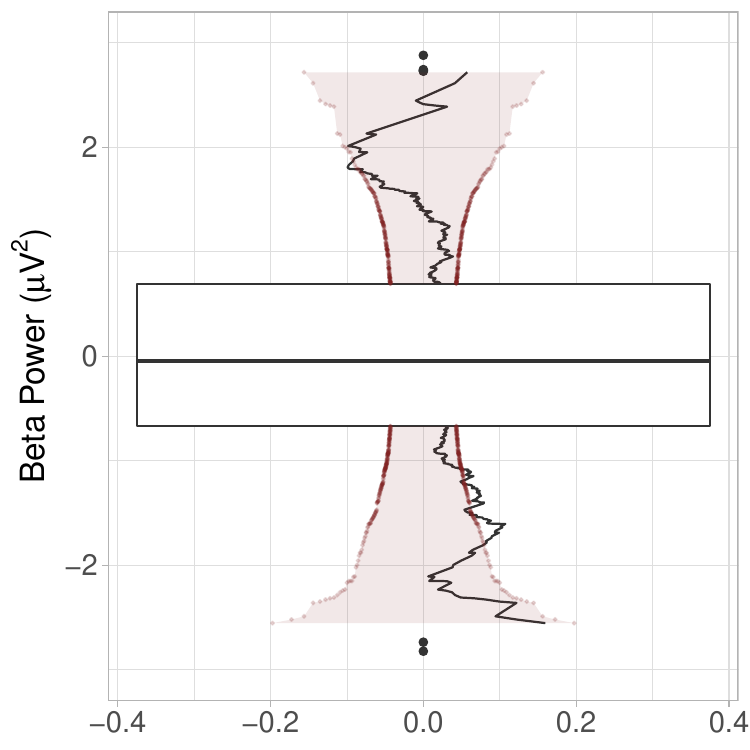} \\
\vspace{-0.2cm}
{\centering \scriptsize \hspace*{.3cm} \textbf{Music} \hspace{6cm} \textbf{ Color}} \\
\centering \includegraphics[scale=0.26]{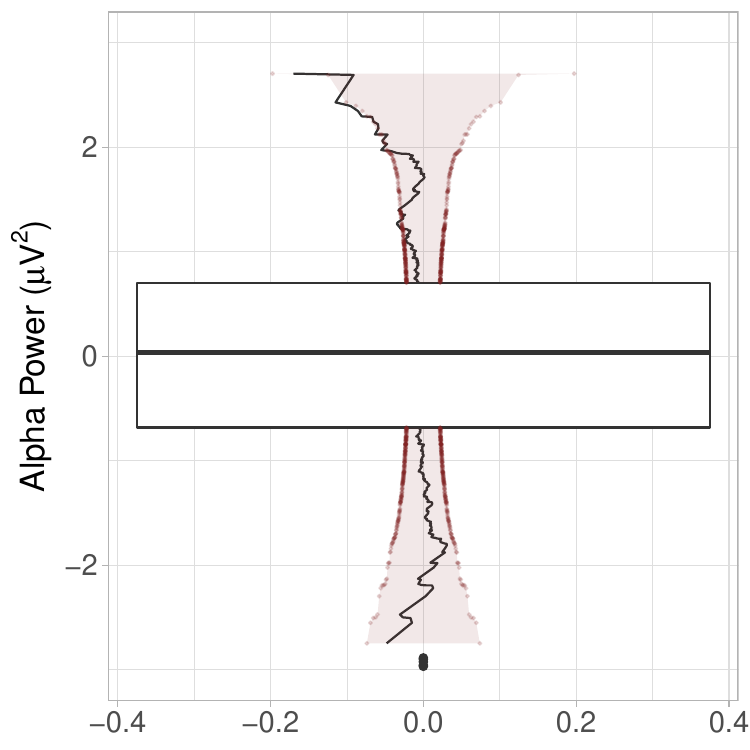}\includegraphics[scale=0.26]{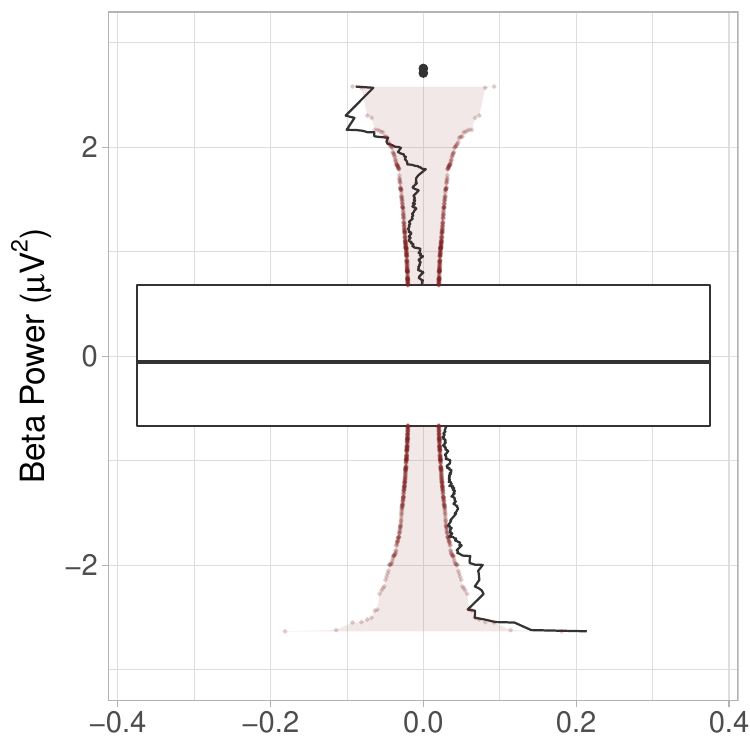}
\centering \includegraphics[scale=0.26]{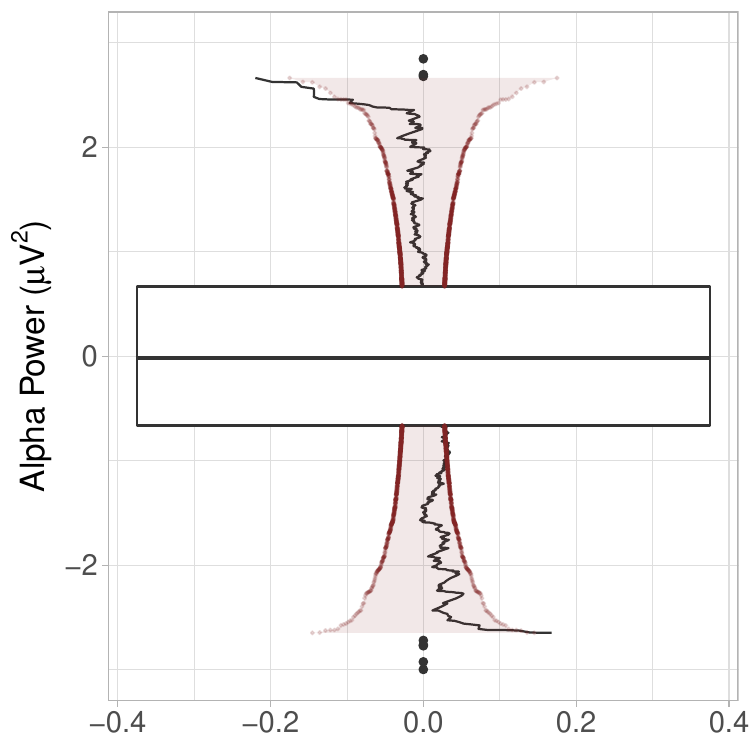}\includegraphics[scale=0.26]{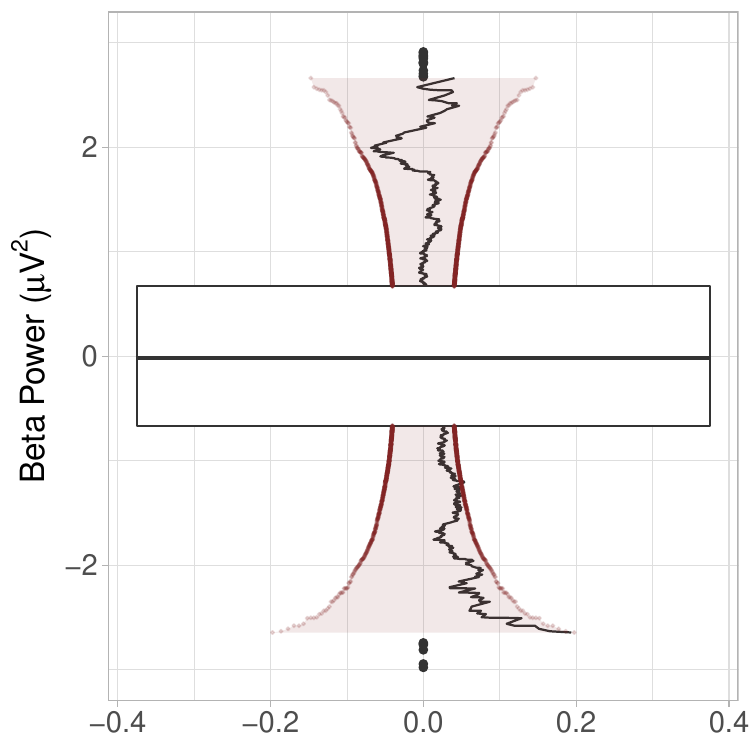} \\
\vspace{-0.2cm}
{\scriptsize \hspace*{.8cm} \textbf{Video}\hspace{5.3cm} \textbf{Relax and think}} \\
\centering \includegraphics[scale=0.26]{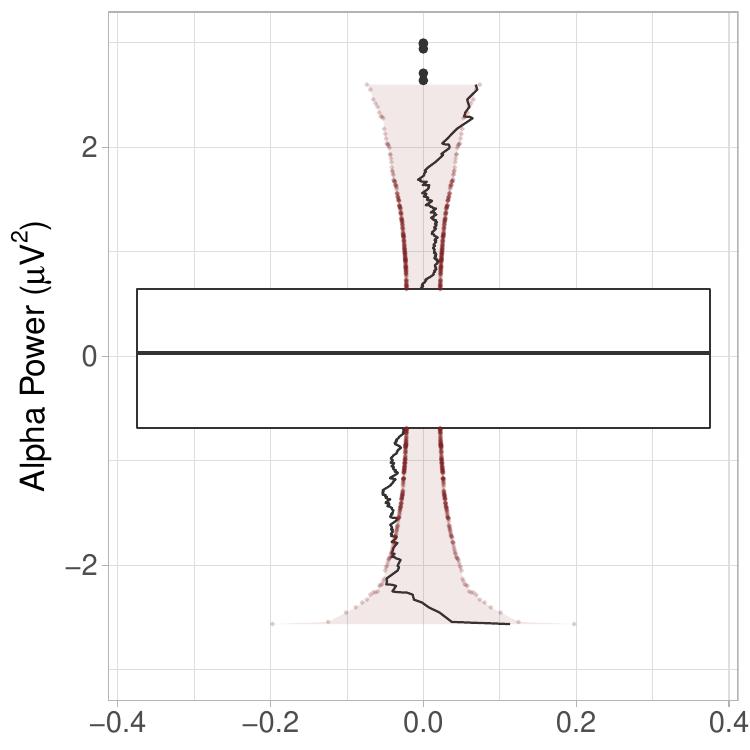}\includegraphics[scale=0.26]{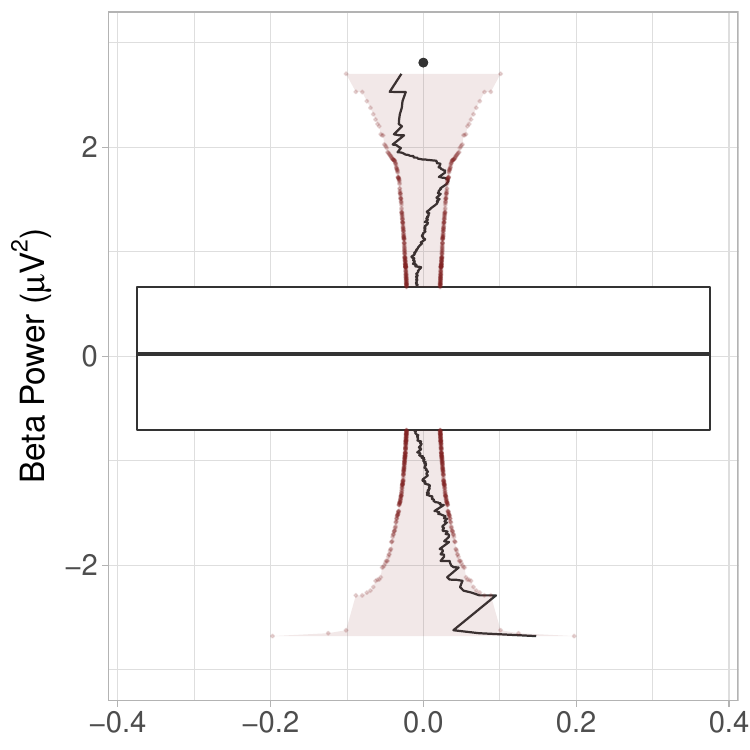}
\centering \includegraphics[scale=0.26]{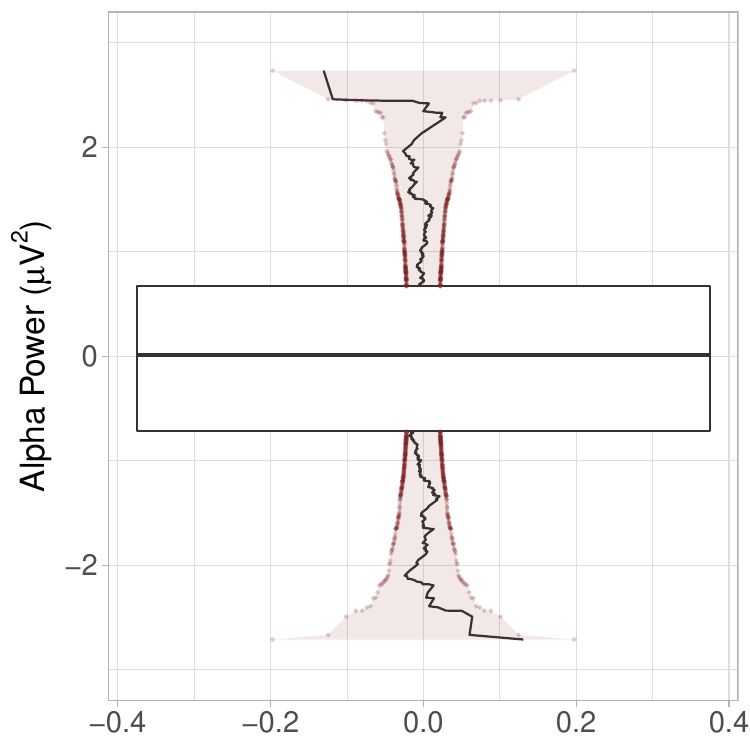}\includegraphics[scale=0.26]{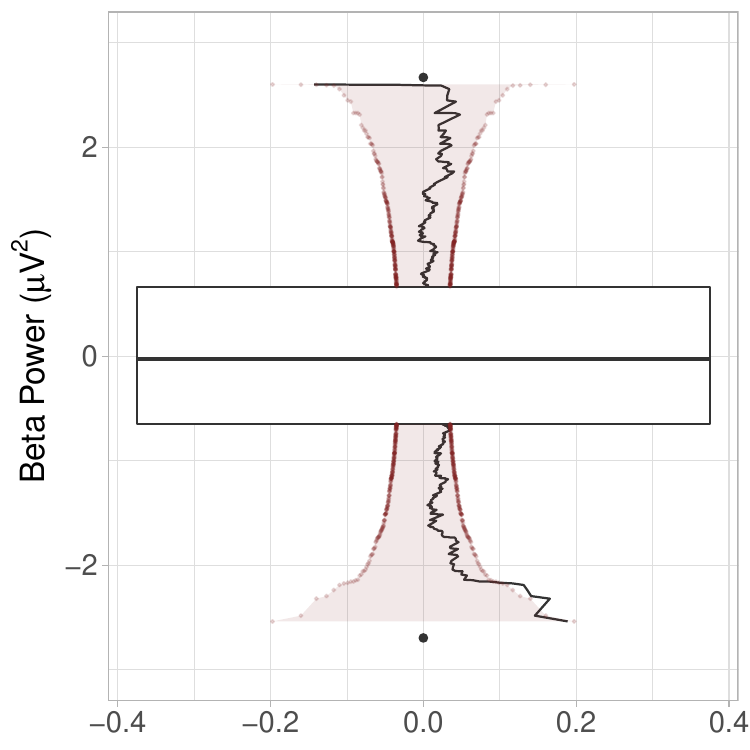}
\caption{ \label{fig:margins} Top: Marginal density estimates
of alpha and beta power for each stimulus, obtained with proposed
stable process scale mixture model from Section~\ref{mixture}, along
with $95\%$ credible bands. Bottom: Corresponding q-q boxplot of randomized quantile residuals.}
\end{figure}
}\vspace{.3cm}

The obtained q-q boxplots provide evidence that the stable process scale mixture adjusts well both the bulk and the right tail of the data. We have also compared the fitted stable process scale mixture model against the DP shape mixtures in \eqref{pymixtures2} with a Pareto kernel and a gamma centering distribution. In line with    the findings from the univariate scenario of the simulation study in Section~\ref{mc}, we again found evidence in favor of a far more sensible behavior of the stable process scale mixture in comparison with the shape mixture of heavy-tailed kernels.  

The comparison of the q-q boxplots in Figure~\ref{fig:margins}---for the stable process scale mixture---against those available from the supplementary material (Section~5)---for the DP shape mixture of Pareto kernels---clearly indicates a better performance of the former over the latter over both the left and right tails.

\subsection{Stimulus-specific joint brainwave analysis}\label{brainwave}

While Section~\ref{mardata} offered a one-dimensional snapshot across different stimuli, we now apply the proposed methods  to learn about the joint distribution of the power of brainwaves on alpha and frequency bands, conditional on the activities and stimuli discussed in Section~\ref{data}. To put it differently, we now apply the multivariate regression framework from Sections~\ref{multivariate}--\ref{conditional} to learn about the conditional dependence structure governing alpha and beta rhythms, and to borrow strength across stimuli, rather than just fitting each density individually as in Section~\ref{mardata}.

\noindent Figure~\ref{fig:stimulus} shows the contours of the fitted conditional joint densities, given the stimulus under analysis, and it sheds light on the dynamics governing the joint behavior of the alpha and beta brain rhythms. First, the joint densities for some stimuli look similar---such as, for example, \textit{music} and \textit{relax and think}---which suggests a similar joint behavior of the rhythms of alpha and beta bands for these stimuli. Second, \textit{mathematics} and \textit{music}---which looked similar just by examining the raw data in Figure~\ref{fig:power} and the marginal fits in Figure~\ref{fig:power}---have a clearly different dependence structure as can be seen from Figure~\ref{fig:stimulus}. In other words, while marginally the alpha- and beta-band oscillations for \textit{mathematics} and \textit{music} do look similar, their `synchronization' or joint behavior looks markedly different. \vspace{-0.2cm}

\begin{figure}[H]
  \renewcommand{\thefigure}{6}
  \footnotesize \textbf{Mathematics} \hspace{2.5cm} \textbf{Relaxation} \hspace{3.1cm} \textbf{Music} \\
\includegraphics[scale=0.33]{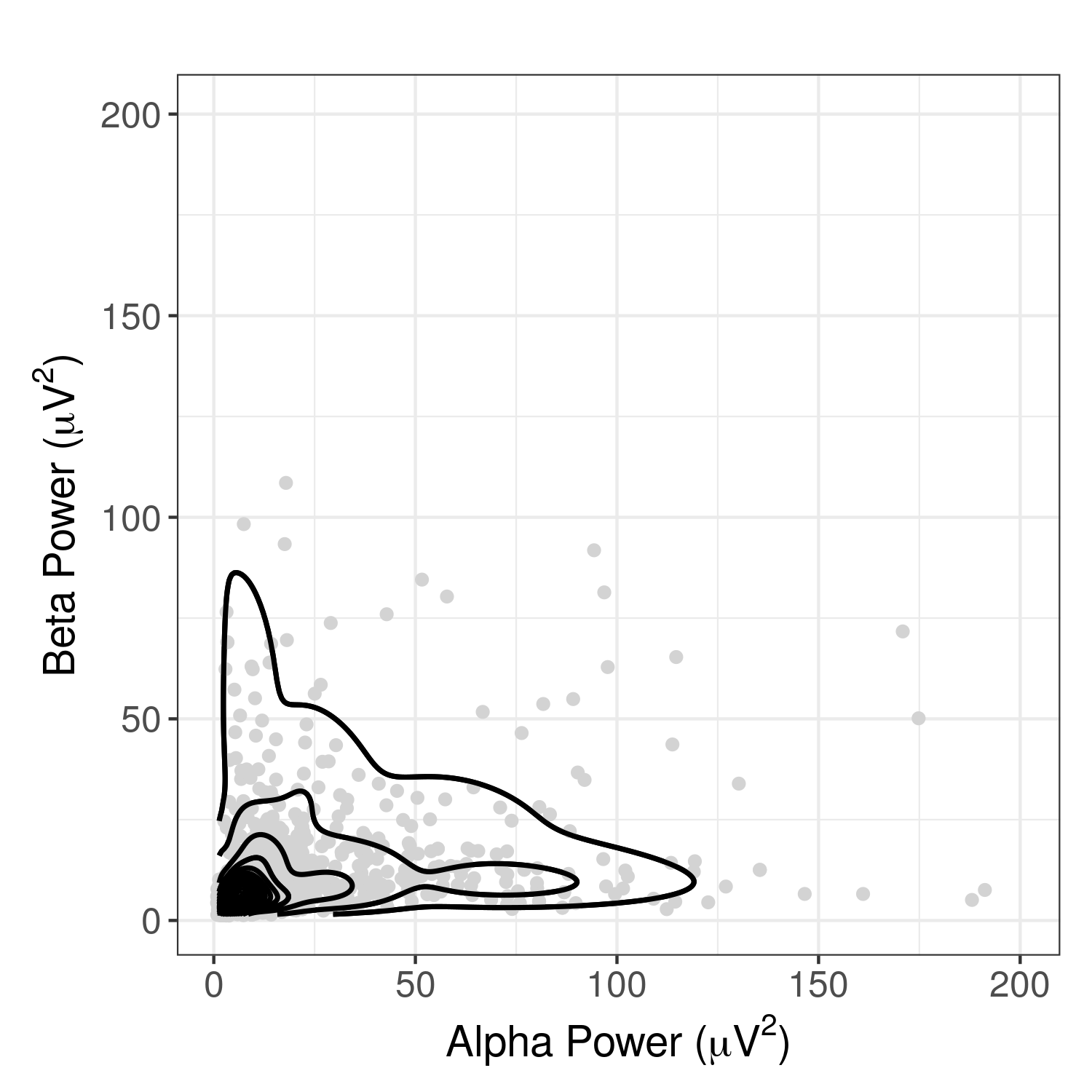}\includegraphics[scale=0.33]{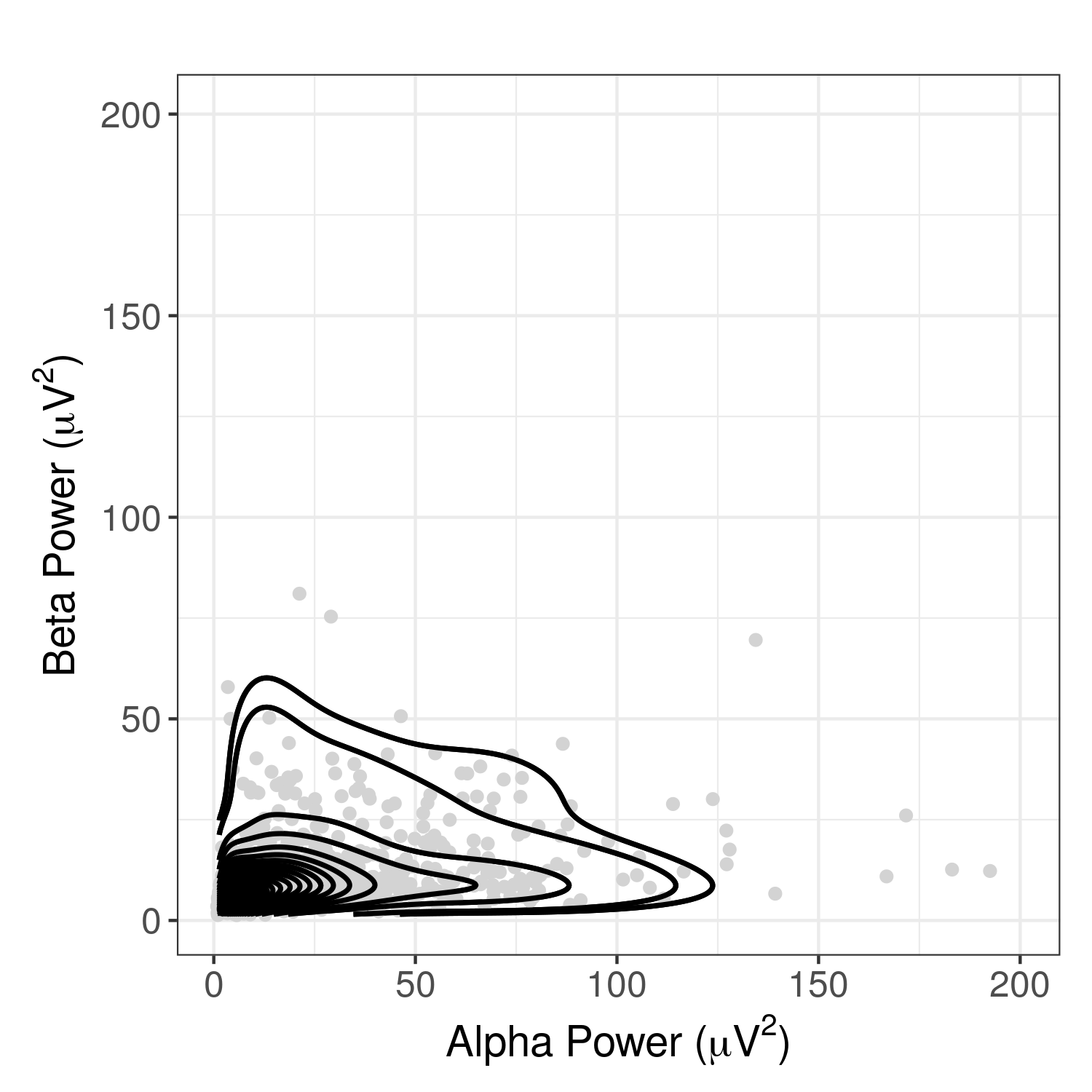}
\includegraphics[scale=0.33]{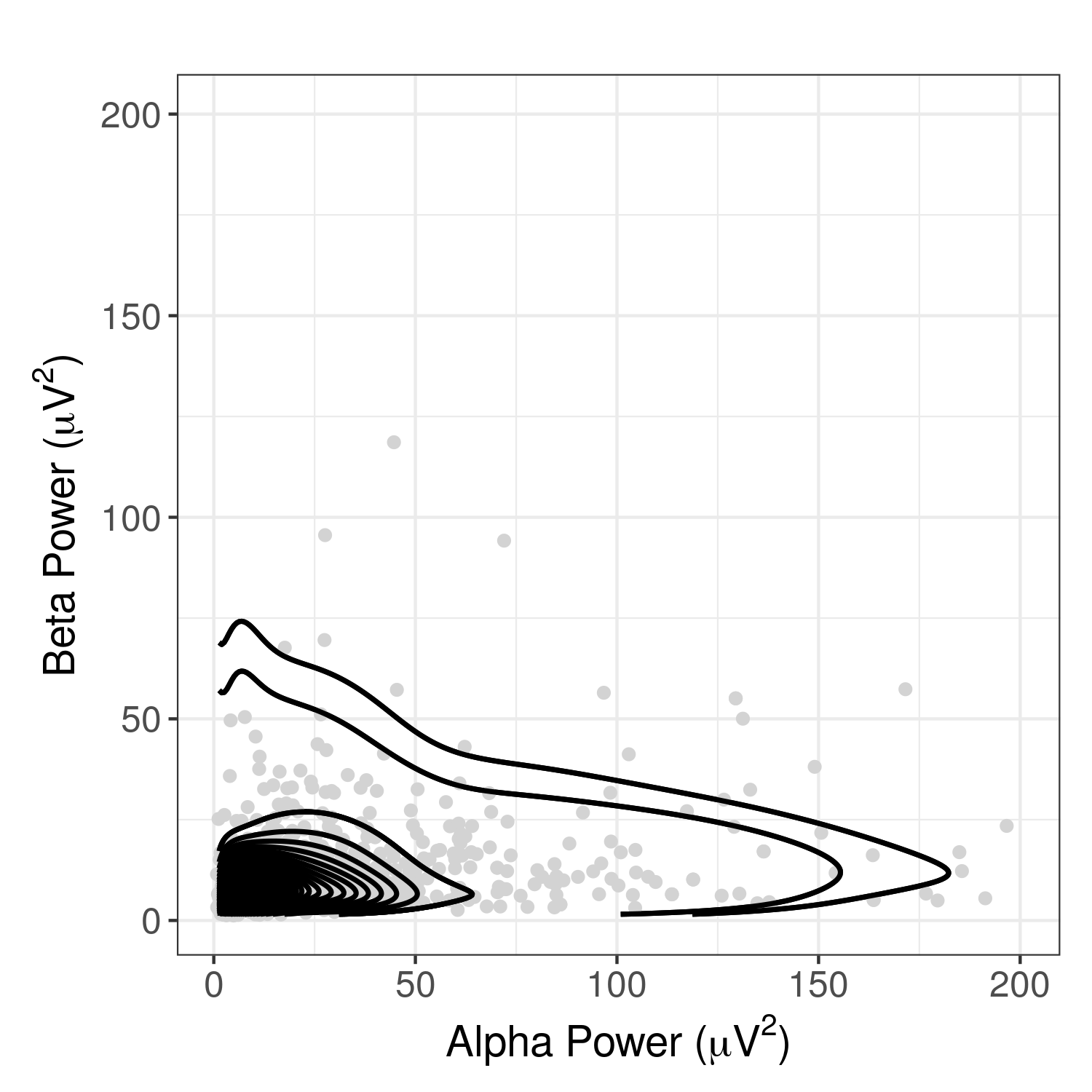}\\
\footnotesize  \hspace{0.8cm} \textbf{Color} \hspace{3.5cm} \textbf{Video} \hspace{2.6cm} \textbf{Relax and think} \\
\includegraphics[scale=0.33]{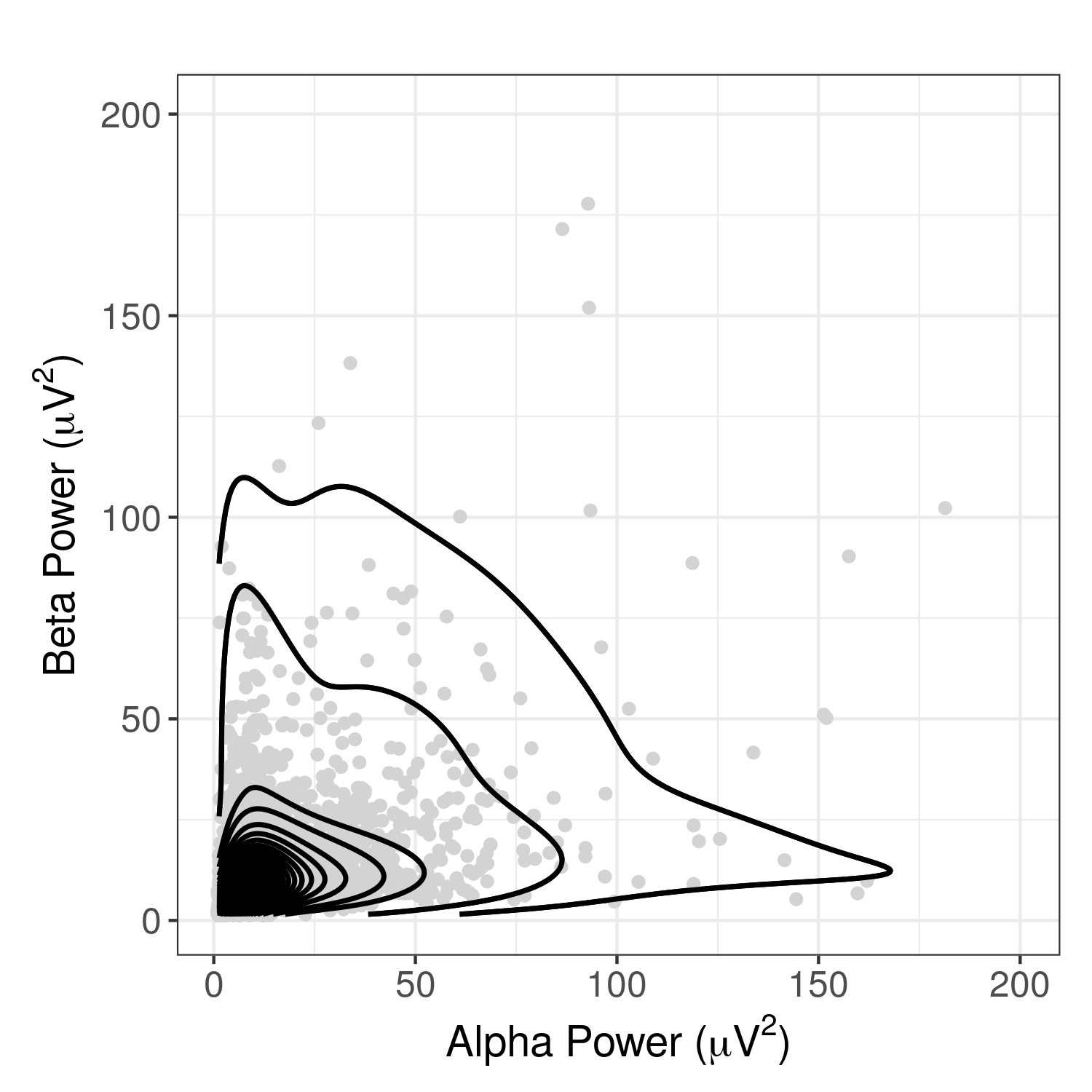}
\includegraphics[scale=0.33]{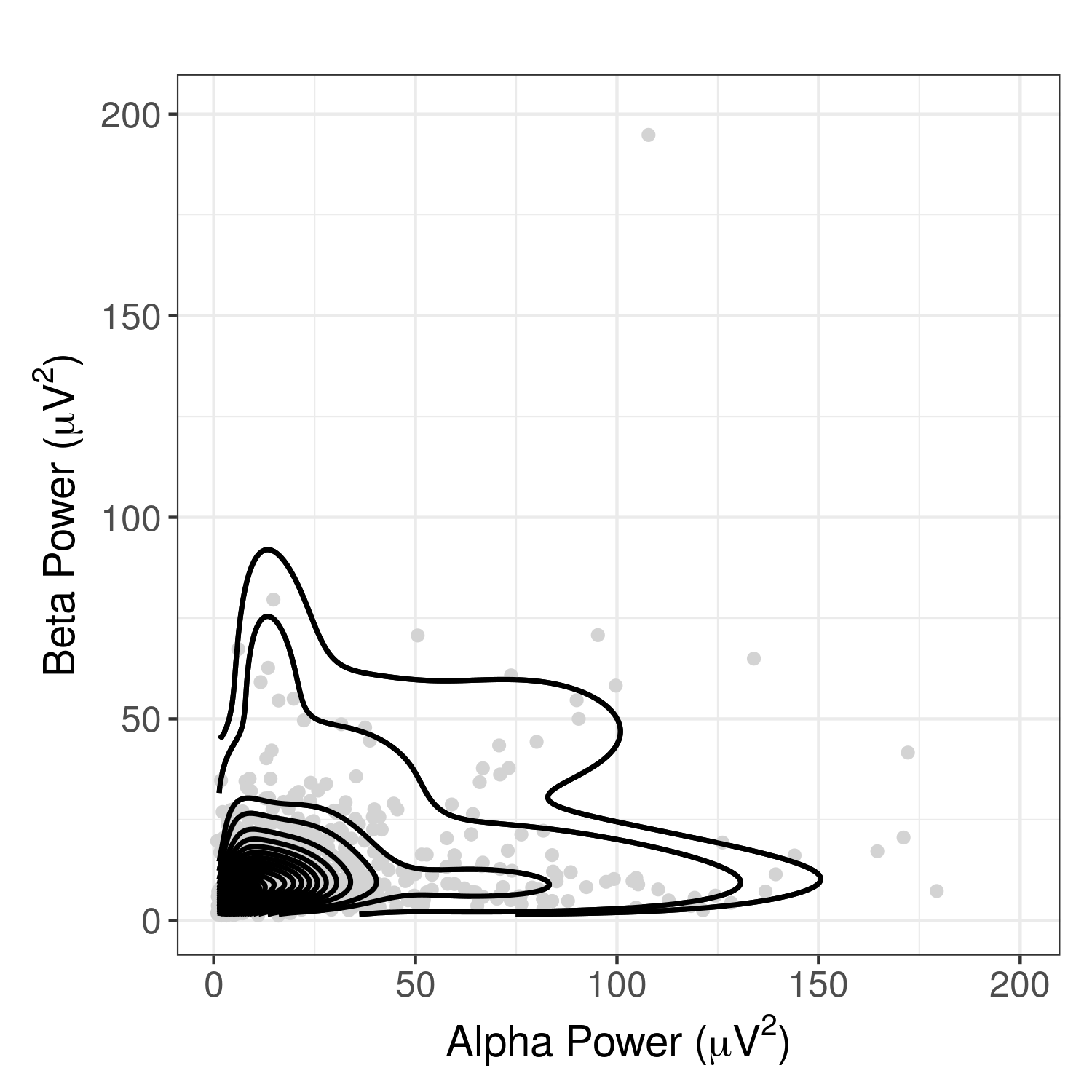}\includegraphics[scale=0.33]{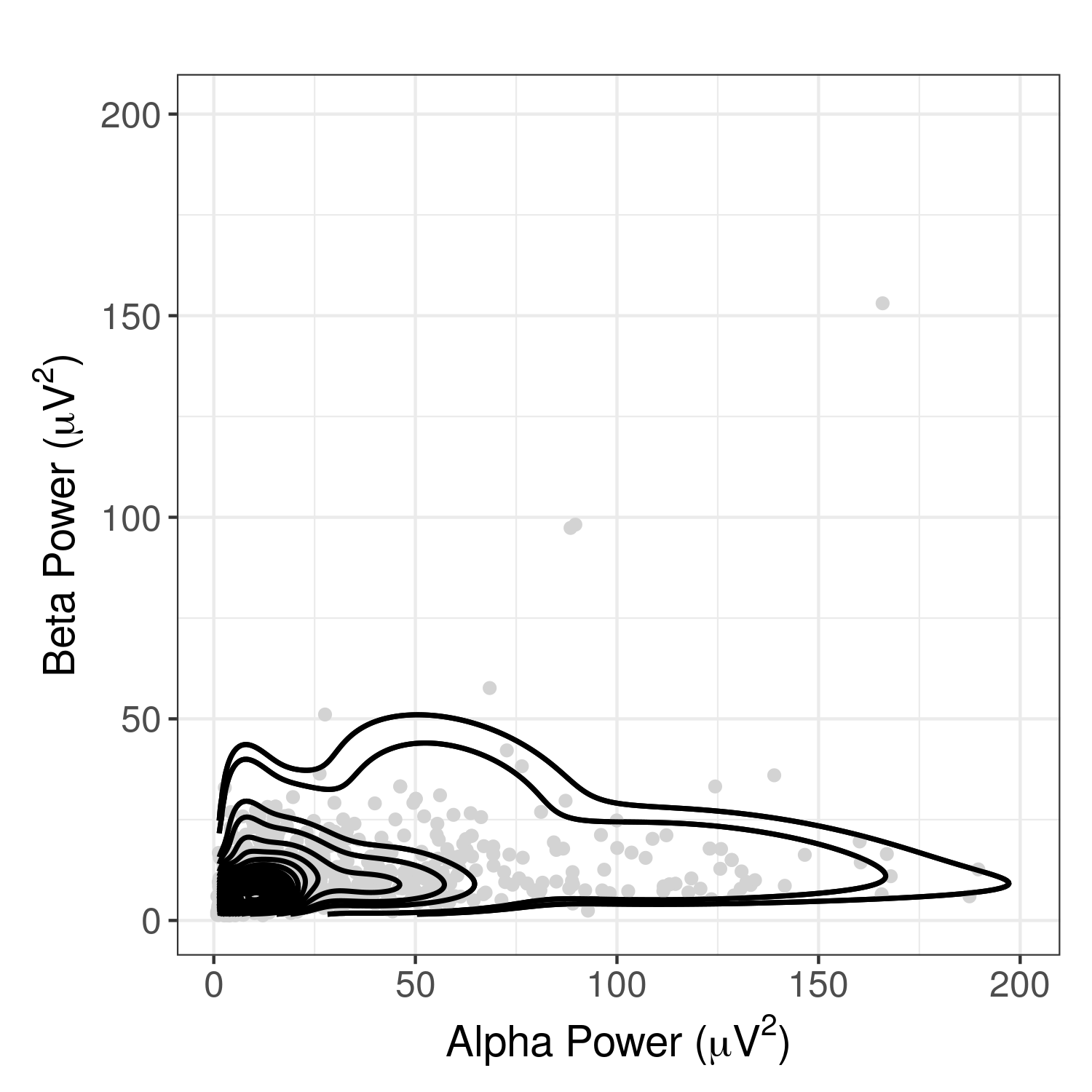} \vspace{-0.3cm}
\caption{\label{fig:stimulus}Contours of the posterior conditional joint density estimate of alpha and beta power for each specific stimulus along with raw data; the fit was obtained using the stable process scale mixture from Section~\ref{conditional}.}
\end{figure}\vspace{-0.2cm}

\section{Concluding remarks}\label{discussion} 

This paper studied the tails of the NGG process, and it has shown in particular that the tail of the NGG in $\mathcal{N}$ is tantamount to that of the baseline. This result is in clear contrast to what is known to hold for the DP, whose tails are exponentially much thinner than those of the centering; in addition, we have also derived for the first time envelopes on which the tail of the NGG-$\mathcal{N}$ must lie. We then devised  two classes of heavy-tailed NGG mixture models, along with their extensions to a multivariate heavy-tailed setting---as well as to a regression framework. Equipped with the above-mentioned characterization of the tails of the NGG process, we have shown that not all heavy-tailed NGG process mixture models are alike. To put it differently, our theoretical and numerical analyses pinpoint a clear preference for NGG-$\mathcal{N}$ scale mixtures over shape mixtures of heavy-tailed kernels. Particularly, we have shown that shape mixtures of Pareto-type kernels can be super heavy-tailed even though the centering is `only' heavy-tailed; this implies that a na\"ive application of the latter mixture models might lead to an overestimation of the mass at the tail---along with poor inferences at the bulk of the distribution. On the contrary, we have found NGG-$\mathcal{N}$ scale mixtures to obey natural properties---such as the stability of the heavy-tail from \red{Corollary~\ref{cor}}---and to perform well numerically in both the bulk as well as in the right tail. Keeping in mind the scope of our case study as well as space constraints, we have concentrated the numerical illustrations on the right tail as well as on kernels supported over the positive real line. Yet, we underscore that the theory and methodologies from Sections~\ref{sect:the_model}--\ref{sect:extensions} hold more generally over the entire real line as well as for left tails. Finally, the fact that other statistical functionals, such as tail indices, can be readily inferred from the proposed methods implies that, as a byproduct, the conditional version from Section~\ref{sect:extensions} may be used as a tail index regression model in the same vein as \cite{wang2009}.

We close the paper with some final comments on open challenges. It is conceivable that similar bounds to the ones derived in this paper could be obtained for the class of processes that can be represented using subordinators. Preliminary derivations lead us to conjecture that a version of Theorem~1 might be possible to obtain for such standardized  subordinators with a slowly varying Laplace exponent, and that a version of Theorem~2 might be possible to obtain for such processes with a regularly varying Laplace exponent (with a non-zero index of regular variation). We leave such an open problem for future analysis. As we have shown here, NGG processes obey the stability of the heavy-tail property, and it would be interesting to have a broader understanding on how large is the class of random probability measures obeying that property. While here the focus has been on the NGG processes and on their mixtures, the potential for modeling  heavy-tailed data of other classes of random measures---such as those of \cite{ayed2019}---remains highly unexplored. Finally, by keeping in mind the importance of modeling rare but catastrophic events in a variety of fields---such as climatology, geology, insurance, risk analysis, and extreme value theory---the methodologies proposed herein may pave the way for further applications and developments at the interface between heavy-tails and Bayesian nonparametrics. \\

The online supplementary materials contain further technical details and proofs, supporting numerical evidence, as well as the R-package \texttt{NGGR} which implements instances of the methods proposed herein, and includes the dataset from Section~\ref{application}.\hspace{-0.2cm}
\section*{Appendix}
\subsection{Technical details and auxiliary lemmata}\label{lemmata} 
In addition to the auxiliary results below, we first recall a basic fact on subordinators  that will be handy for the proof of Theorem~\ref{tailNGG}. If $S(t)$ is a subordinator, then $S({t + h}) - S(t)$ has the same distribution as $S(h)$, for every $t, h \geq 0$ \citep[][p.~5]{bertoin1999}. This implies that $\{S(M) - S(t)\}_{t \in [0, 1]}$ is equal in distribution to $\{S(M - t)\}_{t \in [0, 1]}$, and hence the following subordinator representation holds for the tail of  $G\sim \NGG(M, \tau,D, G_0)$:
\begin{equation}\label{con}
  1 - G(y) = 1 - \frac{S\{MG_0(y)\}}{S(M)}  \overset{\text{d}}{=} \frac{S\{M(1 - G_0(y))\}}{S(M)}.
\end{equation}
\noindent Lemma~\ref{sto} gathers two well-known results on lower and upper envelopes of stochastic processes over the short-run which can be found in \citet[][Theorem~11]{bertoin1999} and \citet[][Theorem~9]{bertoin1999}. Lemma~\ref{rep} is a well-known result in regular variation \citep[e.g.,][Theorem~A.33]{embrechts1997} and for the extended Breiman's lemma see \citet[][Proposition~2.1]{denisov2007}. For completeness, we include in the supplementary material (Section~1) some additional comments on regular variation and on heavy-tails.

\begin{lemma}\label{sto}
  The following results hold:
  \begin{enumerate}[a)]
  \item If $\{S(t): t \geq 0\}$ is a subordinator with Laplace exponent $\Phi \in \RV_{D}$, with $D \in (0, 1)$, then 
      $$\lim \inf_{t \to 0^+} {|S(t)|}/{l(t)} = D (1 - D)^{(1 - D)/D}, \quad a.s.,$$
    where $l(t) =\log |\log t|/ \Phi^{-1}(t^{-1} \log |\log t|)$ for $0 < t < e^{-1}$, and $\Phi^{-1}$ is the inverse function of $\Phi$.
  \item Let $\{S(t)\}$ be a subordinator with drift zero. Let $u(t)$ be an increasing function such that $u(t)/t$ is a real-valued function that is positive, continuous, and increasing, then
\begin{equation*}       \lim \sup_{t \to 0^+} \frac{S(t)}{u(t)} =
    \begin{cases}
      0, & \int_0^{1} \nu[u(t),\infty)\, \dif t < \infty, \\
      \infty, & \int_0^{1} \nu[u(t),\infty) \, \dif t = \infty,
    \end{cases} \quad a.s. \qquad 
\end{equation*}
 
 \end{enumerate}

  \end{lemma}

\begin{lemma}[Representation theorem]\label{rep}
  If $h \in \RV_\alpha$ for some $\alpha \in \mathbb{R}$, then
  \begin{equation*}
    h(y) = c(y) \exp\bigg\{\int_z^y \frac{a(u)}{u}\, \dif u\bigg\}, \quad y > z,
  \end{equation*}
  for some $z > 0$ with $c(y) \to c \in (0, \infty)$, $a(y) \to \alpha$ as $y \to \infty$. The converse also holds. 
\end{lemma}

\begin{lemma}[Extended Breiman's lemma]\label{breiman}
{Let $X$ and $Y$ be random variables and suppose $X$ has a regularly varying tail, $P(X > x) = x^{-\alpha} \mathscr{L}(x)$, with tail index  $\alpha \geq 0$, and $Y \geq 0$ with $E(Y^{\alpha}) < \infty$. Then, if $$\lim \inf_{x \to \infty} \mathscr{L}(x) > 0 \quad \text{and} \quad P(Y > x) = o\{P(X > x)\},$$ it follows that $XY$ has regularly varying tail with tail index $\alpha$.}
\end{lemma}

\subsection{Proofs of main results}\label{proofmr}\vspace{.2cm}
\noindent \textbf{Proof of Theorem~\ref{tailNGG}}: We start with the lower envelope. The Laplace exponent of a generalized gamma process with parameters $D$, and $\tau$, $\Phi(\lambda) = (\lambda+\tau)^{D}-\tau^D$, is regularly varying at $\infty$ with index $D \in (0, 1)$, and note also that $\Phi^{-1}(y) = (y + \tau^D)^{1/D}-\tau$. Hence, Lemma~\ref{sto}~a) implies that
\begin{equation}\label{lic}
    \lim \inf_{t \to 0^+} \frac{S(t)}{l(t)}  = D(1 - D)^{(1 - D)/D}, \quad \text{with}\quad l(t) = t^{1/D} \left\{ \frac{\log |\log t|}{(\log |\log t|+\tau^Dt)^{1/D}-\tau t^{1/D}} \right\}.
\end{equation}
Combining \eqref{lic} with the representation of the NGG-$\mathcal{N}$ process in \eqref{nngg} yields 
\begin{equation*}
  \lim \inf_{G_0(y) \to 0^+} \frac{S\{M G_0(y)\}/S(M)}{l\{MG_0(y)\}}  =
\lim \inf_{G_0(y) \to 0^+} \frac{G(y)}{l\{MG_0(y)\}} = 
D(1 - D)^{(1 - D)/D}/S(M).
\end{equation*}
Hence, \eqref{con} yields
\begin{equation*}
    \lim \inf_{G_0(y) \to 1^-} \frac{1 - G(y)}{l\{M(1 - G_0(y))\}}  = D(1 - D)^{(1 - D)/D}/S(M),
\end{equation*}
from where the final result follows. Next, we focus on the upper envelope. First, consider the following family of functions for $r > 0$, 
\begin{equation}\label{family}
  w_{r}(t) = t^{1/D} |\log t|^{r/D} 
  = t^{1/D} \{\log(1/t)\}^{r/D}, \qquad t \in (0, e^{-r}),    
\end{equation}
where $D \in (0, 1)$. We start by checking if $w_{r}(t)$ verifies the assumptions of Lemma~\ref{sto}~b). 
It follows that $w_{r}(t)$ is increasing on $(0, \delta)$, with $\delta = \text{e}^{-r}$. Indeed, 
\begin{equation}\label{nond}
  \begin{split}
  \frac{\dif}{\dif t}\{w_{r}(t)\} &=
  {D^{-1} t^{1/D - 1}\{\log(1/t)\}^{r/D} + t^{1/D} rD^{-1} \{\log(1/t)\}^{r/D - 1} (-1/t)} \\
  &=
  {D^{-1} t^{1/D - 1}\{\log(1/t)\}^{r/D} - t^{1/D - 1} rD^{-1} \{\log(1/t)\}^{r/D - 1}} \\
  &={D^{-1} t^{1/D - 1}[\{\log(1/t)\}^{r/D} - r \{\log(1/t)\}^{r/D - 1}}] > 0, \quad t \in (0, \delta).
  \end{split}
\end{equation}
In addition, $t^{-1}w_{r}(t)$ is positive and continuous over $(0, 1)$, and increasing on $t \in (0,  e^{r/(D-1)})$ as
\begin{equation*}
  \frac{\dif}{\dif t}\bigg(\frac{w_{r}(t)}{t}\bigg) =
  t^{1/D-2}D^{-1}\{\log(1/t)\}^{r/D-1}\{-r-(D-1)\log(1/t)\} > 0,
\end{equation*}
when $t< e^{r/(D-1)}$. 

The function $w_r(t)$ is increasing in a sufficiently small neighborhood of $0$, and outside this neighborhood it can be extended linearly with slope 0 so that is nondecreasing on $(0,\infty)$, in a similar fashion as  \cite{doss1982}; we denote this extended function by $u_r(t)$. Hence, $u_{r}(t)$ obeys the assumptions of Lemma~\ref{sto}~b). Finally, we derive necessary and sufficient conditions for $\int_{0}^{1}\nu[u_r(t),\infty) \, \mathrm{d}t<\infty$. First, note that 
\begin{align*}
\nu[u,\infty) =\int_{u}^{\infty}\frac{D}{\Gamma(1-D)}x^{-1-D}e^{-\tau x} \,\dif x
\leq \frac{D}{\Gamma(1-D)}\int_{u}^{\infty}x^{-1-D} \, \dif x
=\frac{D}{\Gamma(1-D)}\left\{ \frac{u^{-D}}{D}\right\}, 
\end{align*}
and  
\begin{align*}
\nu[u,\infty) & =\int_{u}^{\infty}\frac{D}{\Gamma(1-D)}x^{-1-D}e^{-\tau x} \, \dif x\\
 & =\frac{D}{\Gamma(1-D)}\left\{ \int_{u}^{1}x^{-1-D}e^{-\tau x}\, \dif x+\int_{1}^{\infty}x^{-1-D}e^{-\tau x} \, \dif x\right\} \\ 
 & \geq\frac{D}{\Gamma(1-D)}e^{-\tau}\int_{u}^{1}x^{-1-D} \dif x\\
 & =\frac{D}{\Gamma(1-D)}e^{-\tau}\left\{ \frac{1}{-D}+\frac{u^{-D}}{D}\right\}.
\end{align*}
Thus,  
\begin{equation}\label{equiv}
\quad\int_{0}^{1}\nu[u_r(t),\infty)\, \dif t<\infty\quad\iff\quad\int_{0}^{1}u_r^{-D}(t) \, \dif t<\infty.
\end{equation}
Hence, applying Lemma~\ref{sto}~b) to the NGG-$\mathcal{N}$ process, with the family of functions on \eqref{family}, yields 
\begin{equation}   \label{cons2}
    \lim \sup_{G_0(y) \to 0^+} \frac{G(y)}{u_r\{G_0(y)\}} =
    \begin{cases}
      0, & \int_0^{1} u_r^{-D}(t) \, \dif t <\infty , \\
      \infty, & \int_0^{1} u_r^{-D}(t)\dif t =\infty.
    \end{cases} \quad \text{a.s.} \qquad 
\end{equation}
Then,  $\int_0^1u_r^{-D}(t) \,\dif t =\int_0^\delta w_r^{-D}(t) \,\dif t+ \int_\delta^1w_r^{-D}(e^{-r}) \, \dif t$, where the second integral is finite, which means $\int_0^1u_r^{-D}(t) \, \dif t $ is finite if and only if $\int_\delta^1w_r^{-D}(t) \,\dif t < \infty$, which holds if $r > 1$. This implies the final result,  
\begin{equation}   \label{cons3}
    \lim \sup_{G_0(y) \to 0^+} \frac{G(y)}{u_r\{G_0(y)\}} =
    \begin{cases}
      0, & r>1, \\
      \infty, & 0< r \leq 1.
    \end{cases} \quad \text{a.s.} \qquad 
\end{equation}

\noindent\textbf{Proof of \red{Corollary~\ref{cor}}}: 
It follows from Theorem~\ref{tailNGG} that as 
$y \to {y^*}$, then
\begin{equation}\label{link}
  1 - G(y) = [M\{1 - G_0(y)\}]^{\{1 + o(1)\} / D}, \quad \text{a.s.}
\end{equation}
By assumption, $1 - G_0 \in \RV_{-\alpha_0}$ and hence it follows by the representation theorem that   
    $1 - G_0(y) = c(y) \exp\{\int_z^y a(u)/u\, \dif u\},$
    for some $z > 0$ with $c(y) \to c \in (0, \infty)$, $a(y) \to -\alpha_0$ as $y \to \infty$. This combined with \eqref{link} yields that
    \begin{equation*}
      1 - G(y) = [M\{1 - G_0(y)\}]^{\{1 + o(1)\}/D} = c^*(y) \exp\bigg\{\int_z^y \frac{a^*(u)}{u}\, \dif u\bigg\}, \quad y > z,
    \end{equation*}
    for some $z > 0$, with
    $$c^*(y) = \{Mc(y)\}^{{\{1 + o(1)\}/D}} \to Mc^{1/D} \in (0, \infty), \quad
    a^*(y) = \bigg(\frac{1 + o(1)}{D}\bigg)a(y) \to -\alpha_0/D,$$ as $y \to \infty$. The final result follows from the representation theorem. \qed

\noindent\textbf{Proof of Theorem~\ref{props}}
\begin{enumerate}[a)]
\item Let $U \mid V = \sigma \sim K(\,\cdot\,; \eta_\sigma)$ and $V \sim G$, and consider the decomposition $U = U_+ - U_{-}$, where $U_+ = \max(U, 0)$ and $U_{-} = \max(-U, 0)$. Since the focus is on the right tail, we concentrate on $U_+$, and note below that for $y > 0$ the scale mixture in \eqref{pymixtures} can be written as the density of the product of $U_+$ and $V$. In detail, it follows from Rohatgi’s well-known result on the product of random variables \citep[e.g.,][]{glen2004}, that 
  \begin{equation}
    \label{brei}
    f_{U_+ V}(y) = \int_{0}^{\infty} f_{U_+, V}\bigg(\sigma, \frac{y}{\sigma}\bigg) \frac{1}{\sigma} \, \dif \sigma.
  \end{equation}
  Hence, combining the fact that $$U_+ \mid V = \sigma \sim K(y; \eta_\sigma) I(y > 0) + \linebreak P(U_+ = 0 \mid V = \sigma) I(y = 0),$$ along with Bayes theorem implies that \eqref{brei} can be rewritten as
  \begin{equation}\label{brei2}
    \begin{split}
      f_{U_+ V}(y) &= \int_0^{\infty} f_{U_+\mid V}\bigg(\frac{y}{\sigma}, \sigma\bigg) \frac{\dif G(\sigma)}{\dif \sigma} \frac{1}{\sigma} \, \dif \sigma \\ &=
      \int_0^\infty K\bigg(\frac{y}{\sigma}; \eta_\sigma \bigg) \frac{1}{\sigma} \,\dif G(\sigma) \\ &= f(y),  
  \end{split}
  \end{equation}
  for $y > 0$.  Since by assumption $G_0$ has a regularly varying tail with tail index $\alpha_0$, it follows from Corollary~\ref{cor} that $V$ has a regularly varying tail with tail index $\alpha_0 / D$. Next, let $V_0 \sim G_0$ and  $\mathscr{L}^*(\sigma) = \{\mathscr{L}(\sigma)\}^{1/D}$ and note that the assumptions along with Corollary~\ref{cor} and the representation theorem imply $E(U_+^{\alpha_0}) < \infty$, $$P(V > \sigma) = \sigma^{-\alpha_0/D} \mathscr{L}^*(\sigma), \qquad P(U_+ > \sigma) = o\{P(V_0 > \sigma)\}^{1/D} = o\{P(V > \sigma)\},$$ as well as that $$\lim_{\sigma \to \infty} \mathscr{L}^*(\sigma) = \{\lim_{\sigma \to \infty} \mathscr{L}(\sigma)\}^{1/D} > 0.$$ In other words, the assumptions of the extended Breiman's lemma apply from where it readily follows that $U_{+} V$ is regularly varying at infinity with tail index $\alpha_0 / D$, and hence the same claim can be made about $f_{U_+ V}(y) = f(y)$. This proves the result.
\item Let $i(h)$ be a permutation such that $\inf\{\alpha: G_0(\alpha) > 0\} \equiv \alpha_{i(1)} \leq \alpha_{i(2)} \leq \cdots$. Then,
  \begin{equation*}
    1 - F(y) = \sum_{h = 1}^\infty \pi_h \frac{\mathscr{L}(y)}{y^{\alpha_h}} 
    = \sum_{j = 1}^\infty \frac{\pi_{i(j)} \mathscr{L}(y)}{y^{\alpha_{i(j)}}} 
    = \frac{\mathscr{L}^*(y)}{y^{\alpha_{i(1)}}} ,
  \end{equation*}
  and it can be easily shown that $\mathscr{L}^*(y) = \mathscr{L}(y) \{\pi_{i(1)} + \sum_{j = 2}^{\infty} {\pi_{i(j)}}/(y^{\alpha_{i(j)} - \alpha_{i(1)}}) \}$ is a slowly varying function, from where the final result follows.
\end{enumerate}

\noindent \textbf{Proof of Theorem~\ref{props_multi}}:  
We only present the proof of Theorem~\ref{props_multi}~a) as that of claim b) follows a similar line of attack. 
We start by showing that the marginal distributions $F_{k}$ are univariate {NGG-mixtures}, 
and then using Theorem~\ref{props} a) it follows that their tails, $1-F_{k}$, are regularly varying, for $k = 1, \dots, d$. 
Let $\dif \mathbf{y}_{-k}=\dif y_{1}\dots\dif y_{k-1} \dif y_{k+1}\dots\dif y_{d}$ and note that \eqref{pymixtures_multi} and \eqref{kerns} implies that 
\begin{equation*}
\begin{aligned}
  f_{k}(y_k)=\int_{\mathbb{R}_+^{d - 1}}f(\mathbf{y})\, \dif \mathbf{y}_{-k}  
 = \int_{\mathbb{R}_+^{d - 1}} \sum_{h = 1}^\infty \pi_h  \mathbf{K}(\mathbf{y}; \etab_{\sigmab_h})\, \dif \mathbf{y}_{-k}
= \sum_{h = 1}^\infty \pi_h K_{\sigma_k}(y_k; \eta_{\sigma_{h,k}})\,.
\end{aligned}
\end{equation*}
Since by assumption $G_{0,k}(\sigma_{k})$ has a regularly varying tail with tail index $\alpha_{0,k}$, it follows from Theorem~\ref{props}~a) that $1-F_k$ is regularly varying with tail index $\alpha(F_k) = \alpha_{0,k} / D$, for  $k=1,\dots,d$, from where the final result follows. \strut \hfill \qed  

\section*{ACKNOWLEDGEMENTS}
We thank the Editor, the Associate Editor, and two Reviewers for their insightful feedback on an earlier draft of this paper. We extend our thanks Isadora Antoniano Villalobos, Vanda In\'acio de Carvalho, and Sara Wade for discussions and constructive comments on an earlier version of the paper. 
The research was supported by the Chilean, Mexican, and Portuguese NSFs through the projects Fondecyt Grant 1220229, ANID--Millennium Science Initiative Program--NCN17$\,$059,\\ https://doi.org/10.54499/UIDB/04106/2020 and https://doi.org/10.54499/UIDP/04106/2020.

\renewcommand\refname{References}
\bibliographystyle{apalike}
\bibliography{test.bib}

\end{document}